\documentclass[12pt,reqno]{amsart}

\newcommand{\SL}{\scal \ell}

\usepackage[active]{srcltx}

\usepackage{latexsym}
\usepackage{graphicx}

\renewcommand{\Re}{{\operatorname{Re}\,}}

\renewcommand{\epsilon}{\varepsilon}

\newcommand{\wt}{\widetilde}

\newcommand{\N}{{\mathbb N}}
\newcommand{\R}{{\mathbb R}}
\newcommand{\C}{{\mathbb C}}
\newcommand{\Q}{{\mathbb Q}}
\newcommand{\Z}{{\mathbb Z}}

\newcommand{\half}{{\textstyle \frac 12}}

\renewcommand{\phi}{\varphi}

\newcommand{\acal}{\mathcal{A}}

\newcommand{\ccal}{\mathcal{C}}

\newcommand{\jcal}{\mathcal{J}}

\newcommand{\lcal}{\mathcal{L}}

\newcommand{\ocal}{\mathcal{O}}
\newcommand{\pcal}{\mathcal{P}}

\newcommand{\rcal}{\mathcal{R}}
\newcommand{\scal}{\mathcal{S}}

\newcommand{\al}{\alpha}

\newtheorem{theo}{{\sc Theorem}}[section]

\newtheorem{mainprob}{{\sc Problem}}
\newtheorem{prob}{{\sc Problem}}

\newtheorem{cor}[theo]{{\sc Corollary}}

\newtheorem{prop}[theo]{{\sc Proposition}}

\title[The inverse spectral problem  ]
{Survey on the  inverse spectral problem  }

\author{Steve Zelditch}
\address{Department of Mathematics, Northwestern  University, Evanston, IL 60208, USA}

\email{zelditch@math.northwestern.edu}

\thanks{Research partially supported by NSF grant DMS-1206527  }

\date{\today}

\begin{document}
\maketitle

Let $(M, g)$ be a  Riemannian manifold  of dimension $n$, possibly with boundary,
and denote its Laplacian by  $$\Delta_g = 
\frac{1}{\sqrt{g}}\sum_{i,j=1}^n \frac{\partial}{\partial
x_i}g^{ij} \sqrt{g} \frac{\partial}{\partial x_j},$$  where    $g_{ij} = g(\frac{\partial}{\partial
x_i},\frac{\partial}{\partial x_j}) $, $[g^{ij}]$ is the inverse
matrix to $[g_{ij}]$ and $g = {\rm det} [g_{ij}].$ When $M$ is compact 
the spectrum of $\Delta_g$ is the discrete set of eigenvalues 
 \begin{equation} \label{EV} \left\{\begin{array}{l}  - \Delta_g \phi_j =
 \lambda_j^2 \phi_j,\;\; \langle \phi_i, \phi_j \rangle = \delta_{i
j}  \\ \\
B\phi_j = 0\;\; \mbox{on}\; \partial M, \end{array} \right.,
\end{equation}
repeated according to their multiplicities. 
When $M$ has a non-empty boundary $\partial M $, one imposes boundary conditions such as 
Dirichlet $B u = u|_{\partial M}$ or Neumann $B u =
\partial_{\nu} u |_{\partial M}$.

More generally, we may consider the eigenvalue problem for  semi-classical   Schr\"odinger operators
\begin{equation} \label{SO} \hat{H}_{\hbar} \psi_{\hbar, j} := \left(- \frac{\hbar^2}{2} \Delta_g + V\right) \psi_{\hbar, j}
= E_j(\hbar) \psi_{\hbar, j}, \end{equation}
where  $V \in C^{\infty}(M)$ is a potential. Here, $\hbar$ is the Planck constant.
When $(M, g)$ is non-compact, the spectrum of $\Delta_g$ is often
continuous, and one might instead pose the inverse problem for  eigenvalues of the scattering operator 
$S_{\hbar} $ (`phase shifts') or poles of its
analytic continuation (`resonances') \S \ref{PHASE}.   If the potential $V$ grows
at infinity, $\hat{H}_{\hbar}$ has a discrete spectrum.

The inverse spectral or scattering  problem is to determine as much as possible of the metric, potential or domain
from spectral or scattering data. In the case of discrete spectrum one would like to ``invert" the 
map $Sp: (M, g, V) \to \mbox{Spec} (\hat{H}_{\hbar}) = \{\lambda_j^2\}_{j = 1}^{\infty}$, or at least to show
that it is 1-1 when restricted to a reasonable class of $(M, g, V)$.    We note that there are two types of inverse problems
for Schr\"odinger operators \eqref{SO}: (i) the more difficult one where $\hbar$ is fixed, and (ii) the substantially simpler one where
one attempts to determine the potential from the 1-parameter family of spectra $\{E_j(\hbar)\}_{j = 1}^{\infty}$.   In some sense, one
would like to use the $\lambda_j$ as ``coordinates'' on the `space' of isometry classes of metrics, potentials or domains.
 This is only a heuristic way  to envision
the problem because of its infinite dimensionality.  Except in dimension one, the range of $Sp$ is not known and is certainly a very small subset of the space $\Sigma \subset \R_+^{\N}$
of sequences compatible with the Weyl law for counting eigenvalues, which in the purely metric
case takes the form $\#\{j: \lambda_j \leq \lambda] \sim C_n Vol(M,g) \lambda^n$.   However it leads one to ask if $Sp$ is `generically' 1-1, to study
its first and second derivative, 
to describe its level sets (isospectral sets),  or  to see if there is any natural additional data one may use to supplement $Sp$ to make it 
a well-behaved map.  A simpler problem is spectral rigidity: can one deform a metric, potential
or domain in a non-trivial way while keeping the eigenvalues fixed?  I.e. do there exist (smooth) curves in an isospectral
set of domains, metrics or potentials with a fixed spectrum?  Isometric domains or metrics are regarded as the
same, or as trivially isospectral,  and when we discuss isospectral domains or metrics it is understood that
they are non-trivially isospectral.

\subsection{Historical background} 

 The inverse spectral problem was first stated by Sir Arthur Schuster in 1882  (see  Appendix F of \cite{St} (p. 395)) :``...Find out the shape of the bell by means of the sound which it is capable of sending out. And this is the problem which ultimately spectroscopy hopes to solve in the case of light. In the meantime we must welcome with delight even the smallest step in the right direction." An early result is that of Rayleigh, determining the coefficients of a differential operator from its eigenvalues \cite{R}.  The problem is best known in the formulation of L. Bers,
S. Bochner,  and M. Kac (``Can you hear the shape of
a drum").  Here, the metric is assumed to be Euclidean, $V = 0$ and one  wants to determine a domain $\Omega \subset \R^2$ from its Dirichlet or Neumann eigenvalues.

In quantum physics, the inverse problem is mainly to determine  the potential (or metric)  of \eqref{SO} from it spectral
and scattering data \cite{N,N2,ChS}.   The idea is that this data is observable in experiments while the potential
is not.  There are many
types of scattering data but in   this survey we only discuss the {\it phase shifts} or eigenvalues  of the scattering
matrix $S_h$  (a unitary operator acting on the sphere; \S \ref{PHASE}). 
The problem of recovering the potential of a Schr\"odinger operator from
eigenvalues and phase shifts was studied since the 1940's  when V. Bargmann discovered two potentials
with the same phase shifts  \cite{Bar1,Bar2}. The inverse problem in dimension one was solved with 
some restrictions by Gel'fand-Levitan and Marchenko  (see \cite{GL,Mar,Mar2} and also \cite{Dy,Dy2} for background and history).  Some other early articles are \cite{A,Bo,Hy,L}.  

The inverse spectral problem in one dimension is simpler than and quite different from the problem in  higher dimensions and
in this survey we only consider
the problem when $\dim M \geq 2$. To give an idea of how the difficulty grows with the dimension, it is simple
to prove that the standard metric $g_0$  on the sphere $S^n$ is determined by the eigenvalues of $\Delta_{g_0}$
when $n = 2$. It is difficult but possible to show  $g_0$ is determined by
$Spec (\Delta_{g_0})$  for $n \leq 6$ (Tanno). It is unknown whether $g_0$ is determined by
$Spec (\Delta_{g_0})$ when $n \geq 7$, although they are known to be spectrally rigid and locally
determined (Tanno).  On the other hand, it is known that balls in $\R^n$ are determined by
their eigenvalues (by the isoperimetric inequality).

There exist two broad classes of results  in inverse spectral theory: (a) construction of isospectral pairs or deformations,
i.e. of `counter-examples' to the inverse spectral problem; (b) positive results proving rigidity or spectral determination
on classes of domains, metrics or potentials. In this survey we concentrate on positive results of type (b) and ask,
how much of the geometry is determined by the spectrum? We refer to \cite{GP} for some results of type (a).

There are a number of recent surveys on the inverse spectral problem  \cite{DH,Z1}. The present survey is
shorter and less systematic. Its purpose is to give a reader who is not very familiar with inverse spectral theory
 an idea of the current state of the art as well as some idea of what is involved in proving
inverse spectral results. Some are new and some are relatively old but illustrate different features
of the problem.  In particular, we do not attempt to provide a comprehensive list of references
(see \cite{DH,Z1}).  One of the charms of the subject is that the problems have an obvious motivation
in both mathematics and physics,  are 
easy to state,  yet remain largely open after 50 years.  
After reviewing the basic techniques we state some tantalizingly simple sounding problems in \S \ref{PROBS}.

\subsection{\label{PRE} Analyticity,  symmetry, accidental degeneracies and multiple spectra}

The inverse spectral problem in higher dimensions involves a vast collection of objects and it is not surprising
that there are few general results applying to all of them. Inevitably there exists sporadic spectral behavior that
allows isospectral sets to have surprising properties. For instance, the eigenvalues or lengths of closed geodesics
might have unusual multiplicities. In the theory of automorphic forms, the Langlands (-Eichler) correspondence
between arithmetic hyperbolic surfaces with cusps and compact arithmetic surfaces defined by quaternion algebras  is a (partial)
isospectrality phenomenon that has no known global geometric interpretation. The correspondence is proved by
`matching orbital integrals', e.g. by comparing lengths of closed geodesics on both surfaces. Thus, the partial isospectrality
has  arithmetic rather than geometric origins.

It  is common to add assumptions or restrictions in order to prove positive results. 
Before introducing definitions and results, we call the reader's attention to some common ones.   \bigskip

\begin{itemize}

\item  (A) The  potentials, metrics or domains are often assumed to have a symmetry, either discrete or continuous. On $\R^n$, potentials or domains are assumed to be invariant under reflections
$x_j \to - x_j$, or under a   di-hedral symmetry  \cite{Z2}.  Continuous symmetry often takes the form of radial potentials, surfaces of revolution etc. The inverse spectral problem is restricted to the  class of such symmetric domains or potentials. \bigskip

\item  (B) The metrics, potentiasl or domains are often assumed to be real analytic. Convexity is another
natural condition. In the counter-example
of \cite{GWW} of non-isospectral domains with the same Dirichlet (resp. Neumann) spectra, the domains have corners
and are non-convex.  \bigskip

\item  (C)The isospectrality is often assumed to occur for more  than one spectrum, as in the case of \eqref{SO} for the one-parameter family $\hbar \to E_j(\hbar)$
of spectra. Or in the case of domains,  one might prescribe both the Dirichlet and Neumann spectra, or both  the spectrum `inside' the domain and also `outside' of it. \bigskip

\item (D) In counter-examples, the metric, potential or domain often has some non-generic behavior such as 
high-dimensional manifolds of closed geodesics or other mulitplicities in the length spectrum. Most positive
results require multiplicity one in the length spectrum.

\end{itemize}  \bigskip

Let us explain the background to these assumptions.

In the early days of one-dimensional inverse spectral theory, Borg \cite{Bo}, Levinson \cite{L} and Marchenko
\cite{Mar,Mar2}   proved that
a potential $V \in L^1[0, L]$ is determined by two spectra of a  Schr\"odinger operator (somewhat similar to 
prescribing spectra of \eqref{SO} for two values of $\hbar$).  When the potential is symmetric, $q(L - x) = q(x)$
then one good spectrum, either  the Dirichlet or Neumann  spectrum, suffices (see e.g. \cite{BrH}). The solution is not unique for general non-symmetric potentials. Results assuming   $\Z_2$ symmetries  are multi-dimensional generalizations of this result. 
See \cite{Z2,GU,He,HZ2} for recent results of this kind. 

Radial symmetry is a very common assumption in inverse scattering theory since it essentially reduces the inverse
problem to one dimension  \cite{KKS,N3}. The author is not aware of any inverse results for phase shifts that do
not assume radial symmetry.  The analogous problem for metrics is to assume the surface has an $S^1$ rotational
symmetry, i.e.  is a surface
of revolution. When it is also mirror symmetric, the inverse problem was solved in \cite{BrH} by reducing the inverse
problem to that of Borg-Levinson-Marchenko.  There are several other results of this kind due to M. Kac, D. Stroock
and P. B\'erard (see \cite{DH}). In \cite{Z3} the author solved the inverse spectral problem for convex analytic
surfaces of revolution without  mirror symmetry; the proof used what we now call a normal forms approach (Theorem \ref{ISPSR}
and \S \ref{SURFREV}).

A plane domain  may be locally defined as the graph of a function $y = f(x)$ over $(-1,1)$ and therefore has the
same number of functional degrees of freedom as a surface of revolution or a radial potential. I.e. the unknown is a function of one
real variable. If the domain has the $\Z_2 \times \Z_2$ symmetries of an ellipse, the problem has the same symmetries
as the one dimensional Schr\"odinger problem of Borg-Levinson. The lack of uniqueness in the 1 D case might suggest lack of uniqueness
for the plane domain problem without two symmetries. However, if the domains are real analytic, the one $\Z_2$ symmetry
is enough to solve the inverse problem for the  Dirichlet or Neumann spectrum (Theorem \ref{ONESYM1D}  \cite{Z2}). It is not known at present if the inverse spectral problem is uniquely solvable (even in the analytic class)  without assuming at least one symmetry.

The history of results for radial potentials or surfaces of revolution  might also suggest that one can only  solve the inverse problem if the unknown is a function of just one real variable, e.g.
that one cannot `hear the shape of a multi-dimensional drum'. However this is also not true. In \cite{HZ} H. Hezari
and the author proved that one can determine a real analytic $\Z_2^n$-symmetric domain in $\R^n$ for any $n$ from its Dirichlet or Neumann
spectrum (Theorem \ref{ONESYM}); here the unknown is a function $f(x_1, \dots, x_n)$
of $n$ variables. An earlier result for the $\hbar$ family of spectra of \eqref{SO} for  $\Z_2^n$-symmetric  potentials was proved in \cite{GU,He}; in \cite{He} fewer  symmetry assumptions were needed. 

Real analyticity is often assumed because the spectral data often leads to local  expressions in terms of Taylor coefficients
of the metric, potential or domain around closed geodesics. Unless one is making a deformation, there is no obvious
way to relate spectral information for distinct geodesics. Real analyticity is used to make a local result global, i.e. to
use Taylor coefficients at one point to determine the unknown function.

The $\hbar$-problem is significantly simpler than the problem for a single spectrum. It is similar to assuming
the knowledge of the spectrum of several commuting operators, e.g. the joint spectrum of $\Delta_g$ and
of $L = \frac{\partial}{\partial \theta}$ on a surface of revolution; here $L$ generates rotations around the axis. 
In \cite{Z2}, the principal step is to show that the single spectrum of $\Delta_g$ determines the joint spectrum. 
This is clearly necessary if the inverse spectral  problem for one value of $\hbar$ is solvable.

Regarding (D), we recall that isospectrality of two Laplacians $\Delta_j$  is the condition that there exists a unitary 
operator $U$ such that $U \Delta_1 U^* = \Delta_2$. In the correspondence principle between classical
and quantum mechanics, the  classical condition is
that there should exist a symplectic diffeomorphism $\chi$ between the cotangent bundles which
restricts to a diffeomorphism  $\chi: S^*_{g_1} M_1 \to S^*_{g_2} M_2$  between the unit cosphere bundles
and which conjugates the geodesic flows $G_j^t$ in the sense that $\chi G_1^t \chi^{-1} = G_2^t$.  The
correspondence principle only works if $U$ is a special type of operator known as a Fourier integral operator. We
refer to \cite{Zw} for background. In \cite{Z5} it was pointed out that the symplectic maps underlying unitary
Fourier integral operators can be multi-valued correspondences, and that all of the well-known counter-examples
of Sunada \cite{Su} were conjugate by unitary Fourier integral operators (the unitary conjugation was also observed
by B\'erard, who termed them transplantations).   Thus, `most' of the counterexamples are Fourier-isospectral
but there are some examples which are not. 

 As discussed in \S \ref{BIRK}, one of the main approaches
to inverse results is through Birkhoff canonical forms, which are special `local'  unitary Fourier integral conjugations
between Laplacians in small neighborhoods of closed geodesics. This is a different kind of condition from 
local isometry of the metrics but is related to it in special cases. See \cite{Gor,Sch} for discussion of isospectral
pairs which are not locally isometric.

As this discussion indicates, a natural direction for future work in inverse spectral theory is to 
 remove the symmetry and analyticity assumptions as much as possible. It would be
useful  to know if (or to what extent)  symmetry or analyticity were a spectral invariant for some reasonable classes of metrics, potentials
or domains. The only result of this kind known to the author is the recent one  of \cite{DHV}, which shows that, in the case of
the 1-parameter $\hbar$ inverse spectral problem,  a radial monotonic potential $V_0$ of a Schr\"odinger operator \eqref{SO} on $\R^n$ is
spectrally determined among `all' smooth potentials (not assuming any symmetry of the latter). 

With this background in mind we now present the basic objects in inverse spectral theory.

\section{Spectral invariants}

The data from which one hopes to recover the metric or potential is the  list $\{\lambda_j^2\}_{j = 0}^{\infty}$ of eigenvalues,
enumerated with multiplicity, i.e. the dimension of the eigenspace \eqref{EV}. We first consider how to assemble
the data into useful invariants. 

\subsection{Traces and trace formulae}

Since the origins of inverse spectral theory, 
the principal spectral invariants are  defined by trace formulae. One assembles the eigenvalue
data into generating functions such as the heat trace $\Theta(t): = \sum_j e^{-t \lambda_j^2}$ or wave trace
\begin{equation} \label{WT} S(t) : = \sum_j e^{i t \lambda_j}. \end{equation} The simplest way to relate spectrum
and geometry is to study the  singularities of these spectral functions. 

The heat trace is smooth except at $t = 0$, and the coefficients of its expansion at $t = 0$ were the first
spectral invariants to be studied. In dimension 2, the heat trace has the form $\Theta(t) = t^{-1} \mbox{Area}(M,g)
+ \frac{2\pi }{3} \chi(M) + \frac{t}{60} \int_M \tau_g^2 d V_g + \cdots$. There is just one singular term and its
coefficient is the integral of a local geometric invariant. One also sees that all of the Taylor coefficients of $t \Theta(t)$
are given by integrals of local geometric invariants. This shows that there are useful spectral invariants besides
the ones which arise from singularities.

\subsection{Length spectrum and wave trace singularities}

The {\it length spectrum} of a boundaryless manifold $(M, g)$ is
the  set
 \begin{equation} Lsp(M,g) = \{ L_{\gamma_1} < L_{\gamma_2} < \cdots\} \end{equation}
 of lengths of closed geodesics $\gamma_j$, i.e.   the  set
of
 distinct
 lengths, not including multiplicities. One refers to the the
 length spectrum repeated according to multiplicity as the
 {\it extended length spectrum}.

For a generic metric on a manifold without boundary, the length spectrum is a discrete set
and moreover the lengths have multiplicity one.  In the boundary case, the length spectrum
 $Lsp(\Omega)$  is the set of lengths of closed billiard
 trajectories, which has
 accumulation points  at lengths of closed billiard  trajectories which glide for some interval of time
 along the boundary. In the case of convex plane domains, e.g.,
 the length spectrum is the union of the lengths of periodic
 reflecting rays and multiples of $|\partial \Omega|.$

\subsection{Singular support of the wave trace}

The wave group of $(M, g)$ is the unitary group 
\begin{equation} \label{Ut} U(t) = e^{i t \sqrt{-\Delta_g}} : L^2(M) \to L^2(M). \end{equation}  Its
trace \eqref{WT}  is a distribution on $\R$. 
The first result on the wave trace is the Poisson relation on a
manifold without boundary, stating that the singular support of $S(t)$ is contained in the length spectrum,
\begin{equation}\label{SS} \mbox{Sing Supp}\; S(t) \subset \;\;
Lsp(M,g).
\end{equation}
It was proved by Y. Colin de Verdi\`ere \cite{CdV3}, Chazarain
\cite{Ch2}, and Duistermaat-Guillemin \cite{DG} (following
non-rigorous work of Balian-Bloch \cite{BB2} and Gutzwiller).  The generalization to manifolds with boundary was
proved by Anderson-Melrose \cite{AM} and Guillemin-Melrose
\cite{GM}. As above, we denote the length of a closed geodesic
$\gamma$ by $L_{\gamma}.$ For each $L = L_{\gamma} \in Lsp(M,g)$
there are at least two closed geodesics of that length, namely
$\gamma$ and $\gamma^{-1}$ (its time reversal). The singularities
due to these lengths are identical so one often considers the even
part of $S(t)$ i.e. $Tr E(t)$ where \begin{equation} \label{Et} E(t)= \cos (t
\sqrt{-\Delta_g}).\end{equation}
The containment relation   (\ref{SS})  could be strict if 
a length $L \in Lsp(M, g)$ is multiple. In this case,  $S(t)$  might be
smooth at $L \in Lsp(M, g)$.
\medskip


One cannot determine multiplicities of lengths  from the singularities of $S(t)$ in any simple way. 
  The  Sunada-type
isospectral pairs \cite{Su}  always have multiple length spectra, and for
many (presumably, generic) examples, the length spectra have
different multiplicities.

 In the notation for $Lsp(M, g)$ we wrote $L_{\gamma_j}$ as if the
 closed geodesics of this length were isolated. But in many
 examples (e.g. spheres or flat tori), the geodesics come in
 families, and the associated length $T$ is the common length of
 closed geodesics in the family. In place of closed geodesics, one
 has
 components of the fixed point sets of $G^T$ at this time. 


\subsection{Singularity expansions}

The next fact is that $S(t)$ has a
singularity expansion at each $L \in Lsp(M,g)$:
\begin{equation} \label{WTE} \begin{array}{l} S(t)  \equiv  e_0(t) + \sum_{L \in Lsp(M,g)} e_L(t)\;\; mod \;\; C^{\infty},
\end{array} \end{equation} where $e_0, e_L$ are Lagrangean distributions
with singularities at just one point, i.e. $\mbox{sing supp}\; e_0 = \{0\},
\mbox{sing supp}\; e_L = \{L\}$. When the length functional on the
loopspace of $M$ is a Bott-Morse functional, the terms have
complete asymptotic expansions.  In the Morse case (i.e. bumpy
metrics), the expansions take the form \begin{equation} e_0(t) =
a_{0,-n}(t+i0)^{-n} + a_{0, -n+1}(t+i0)^{-n+1}+\cdots
\end{equation}
\begin{equation}  \label{WEXP} \begin{array}{lll}
e_L(t) &=& a_{L,-1} (t-L+i0)^{-1} + a_{L,0}\log (t-(L+i0))\\[10pt]
&+&a_{L,1} (t-L+i0)\log (t-(L+i0)) +\cdots\;\;,\end{array}
\end{equation} where $\cdots$ refers to homogeneous terms of ever
higher integral degrees ([DG]). The wave coefficients $a_{0,k}$ at
$t=0$ are essentially the same as the singular heat coefficients,
hence are given by integrals over $M$ of $\int_M P_j(R,\nabla
R,...)\mbox{dvol}$ of homogeneous curvature polynomials. The wave
invariants for $t \not= 0$ have the form:
\begin{equation} \label{WINV} a_{L, j} = \sum_{\gamma: L_{\gamma}
= L} a_{\gamma, j}, \end{equation} where $a_{\gamma, j}$ involves
on the germ of the metric along $\gamma$. Here,  $\{\gamma\}$ runs
over the set of closed geodesics, and where $L_\gamma$,
$L_\gamma^{\#}$, $m_\gamma$, resp.\ $P_\gamma$ are the length,
primitive length, Maslov index and linear Poincar\'e map of
$\gamma$. The Poincar\'e map is defined in \eqref{P}.  For instance,
 the
principal wave invariant at $t = L$ in the case of a
non-degenerate closed geodesic is given by
 \begin{equation} \label{PRIN} a_{L,-1} = \sum_{\gamma:L_{\gamma}=L}
\frac{e^{\frac {i\pi }{4} m_{\gamma}
}L_\gamma^{\#}}{|\det(I-P_{\gamma})|^{\half}}. \end{equation} The
same formula for the leading singularity  is valid for periodic
reflecting rays of compact smooth Riemannian domains with boundary
and with Neumann boundary conditions, while in the Dirichlet case
the numerator must be multiplied by $(-1)^r$ where $r$ is the
number of reflection points (see \cite{GM,PS}).

The wave invariants for $t \not= 0$ are both less global and more
global than  the heat invariants.  First, they are more global in
that they are not integrals of local invariants,  but involve the
semi-global first return map of the closed geodesic  ${\mathcal P}_{\gamma}$. One could
imagine different local geometries producing the same first return
map. Second, they are less global because they are determined by
the germ of the metric at $\gamma$ and are unchanged if the metric
is changed outside $\gamma$.

Thus, associated to any closed geodesic $\gamma$ of $(M, g)$ is
the sequence $\{a_{\gamma^r, j}\}$ of wave invariants of $\gamma$
and of its iterates $\gamma^r$. These invariants depend only on
the germ of the metric at $\gamma$. The principal question of this
survey may be stated as follows:

\begin{mainprob}\label{WI}  How much of the local geometry   of the metric $g$ at
$\gamma$ is contained in the wave invariants $\{a_{\gamma^r,
j}\}$? Can the germ of the metric $g$ at $\gamma$ be determined
from the wave invariants? At least, can the symplectic equivalence
class of its germ be determined?
\end{mainprob}

 As will be discussed below,
the classical Birkhoff normal
form of the metric (or the Poincar\'e map $P_{\gamma}$)
at $\gamma$ is determined by the wave trace invariants  \cite{G, G2, Z3, Z4} .

\subsection{Domains and metrics with the same wave invariants}

Although we are emphasizing wave invariants, it is known that they do not generally determine a domain or metric. 
There exists an example known as a  Penrose mushroom  (due to Michael Lifshits) which shows  that wave invariants are not
sufficient to discriminate between all pairs of smooth billiard
tables. Indeed, Lifshitz constructs (many) pairs of smooth domains
$(\Omega_1, \Omega_2)$  which have the same length spectra and the
same wave invariants at corresponding pairs of closed billiard
orbits $\gamma_j$ of $\Omega_j$ ($j = 1,2$). We refer to \cite{FK} for proof that the domains with the same wave invariants
in general have no eigenvalues in common. One can prove this by contradiction using a deformation argument;
see also \cite{GH} for the generalization to Schr\"odinger  operators/

The domains are constructed in part from ellipses, which have very special billiard dynamics, and are thus far
from generic. They are also not real analytic. Thus they exemplify the strange behavior that can result in wildly
non-generic cases mentioned in (D) of \S \ref{PRE}. It would be interesting to find spectral invariants that distinguish
them.

\subsection{Residual spectral invariants and the Poisson-Wave trace}

More generally one can consider zeta functions such as \begin{equation} \label{Stz} S(t, z) = \sum_{j = 1}^{\infty} \lambda_j^{-z} e^{i t \lambda_j}, \;\; (\Re z >> 0)\end{equation}
and study the poles and residues of its analytic continuation in $z \in \C$. The {\it residual spectral invariants} are the
residues at the poles. They  have local geometric expressions,
which relate spectrum and geometry.  In particular the {\it wave invariants} above are residues of  $S(t, z)$ at its poles in $z$
for
$t \not= 0$.
 Note that \eqref{WT} is the boundary value of the holomorphic function
\begin{equation} \label{PT} S(t + i \tau) : = \sum_j e^{i (t + i \tau)  \lambda_j}, \;\;\; \tau > 0 \end{equation}
in the upper half plane $\{t + i \tau: \tau > 0\}$. More generally one may define the two-variable holomorphic function
 $S(t + i \tau, z)  =  \sum_j  \lambda_j^{-z} e^{i (t + i \tau)  \lambda_j},$ and study its (branched) meromorphic continuation
from $\C_+ \times \C_+\; (\Re \tau > 0, \Re z >> 0)$ to $\C \times \C$. There may exist new spectral invariants
from this analytic continuation. 

\subsection{Non-residual or non-local spectral invariants}

There exist many results on such non-local spectral invariants as the lowest eigenvalue $\lambda_1^2$ or
the Laplace determinant  $\det \Delta_g$. 
One can also consider special values of $\Theta(t)$ or $S(t,z)$ at non-singular values. Non-local invariants are rarely  computable
in geometric terms. Examples which are  the coefficients of positive integral powers of $t$ in the heat
trace $Tr e^{- t \Delta_g}$ at $t = 0$.  In this survey we concentrate on residual spectral invariants.

\section{Explicit formulae for wave invariants }

To explain what is meant by `local geometry' we now give the formula for the `wave trace invariants' or
(equivalently) Birkhoff canonical form invariants around a closed geodesic $\gamma$ in the two-dimensional
case. 

\subsection{The principal wave trace invariant}

A lot of information about $(M, g)$ is already contained in the principal wave invariant \eqref{PRIN}. In \cite{DG}
(and its Appendix) it is proved that when $Lsp(M, g)$ is simple, one may recover the lengths of all closed
geodesics and the eigenvalues of all of the Poincar\'e maps $P_{\gamma}$. This suggests already that the geodesic
flows of isospectral Riemannian manifolds should not be far from symplectically equivalent in the sense
discussed in \S \ref{PRE},  at least in neighborhoods in $S^*_g M$ of closed geodesics.

\subsection{Riemannian manifolds without boundary}

The wave
invariants at a closed geodesic  $\gamma$ are  invariants of the
germ of the metric  at $\gamma$.  We introduce   Fermi normal coordinates $(s,y)$ along
$\gamma$ and denote the corresponding  metric coefficients by
$g_{ij}$. That is, the coordinates $(s, y)$ correspond to $\exp_{\gamma(s)} y \nu_{s}$ where
$\nu_s$ is the unit normal at the point $\gamma(s)$.


We then consider orthogonal Jacobi fields along $\gamma$, i.e. normal
fields $Y(s) = y(s) nu_s$ where $Y$ solves the  Jacobi equation $Y'' + R(Y, \gamma') \gamma'= 0$
along $\gamma$. The space of complex Jacobi fields along $\gamma$
is denoted ${\jcal}_{\gamma}$. The linear Poincare map is the
monodromy map 
\begin{equation} \label{P} P_{\gamma}: Y(t) \to Y(t + L_{\gamma}) \end{equation} on this space. We refer
to its eigenvectors $Y_j, \overline{Y}_j$ as Jacobi eigenvectors. 

We  assume that the
geodesic is  non-degenerate in the sense that   $(\det (I - P_{\gamma})
\not= 0)$.  To simplify the exposition, we only consider the case
where $\gamma$ is elliptic, i.e.  where $P_{\gamma}$ has unit modulus eigenvalues $e^{i \alpha_j}$. 
We define the Floquet invariants
$$\beta_j = (1- e^{i\alpha_j})^{-1}.$$
There is a similar story for hyperbolic closed geodesics.




\begin{theo}\label{BNFDATA}  \cite{Z4} Let ${\gamma}$ be a strongly
non-degenerate closed geodesic.  Then the $k$th wave invariant at $\gamma$ has the form  $$a_{\gamma k} =
\int_{\gamma} I_{\gamma; k} (s; g)ds$$ where:

\noindent(i)  $I_{\gamma; k}(s; g)$  is a homogeneous 
 polynomial  of weight $-k-1$ (under the scaling $g \to \epsilon^2 g$)  in the data
$\{ y, \dot{y}, D^{m}_{s,y}g \}$ with $m=(m_1,\dots,
m_{n+1})$ satisfying $|m| \leq 2k+4$ ;

\noindent(ii) The degree of $I_{\gamma; k }$ in the Jacobi field
$Y = y \nu$ is at most $6k+6$;

\noindent(iii) At most $2k+1$ indefinite integrations over
$\gamma$ occur in $I_{\gamma; k }$;

\noindent(iv) The degree of $I_{\gamma; k }$ in the Floquet
invariants $\beta_j$ is at most $k+2$.\end{theo}

The formula is simplest in dimension $2$, where there is only one
Floquet invariant $\beta$. We use the notation   $\tau$ for the
scalar curvature,  $\tau_{\nu}$ for its unit normal derivative
along $\gamma$, $\tau_{\nu \nu}$ for  the Hessian ${\rm
Hess}(\tau)(\nu,\nu)$. We denote by
 $Y$  the unique
normalized Jacobi eigenvector along $\gamma$ and by  $\dot{Y}$ its
time-derivative.
The subprincipal  wave invariant $a_{\gamma 0}$ is then given by:
\begin{equation} a_{\gamma 0} = \frac{a_{\gamma, -1}}{L^{\#}}[ B_{\gamma 0;4} (2\beta^2 - \beta
-\frac{3}{4}) + B_{\gamma 0;0} ]\end{equation} where:

\noindent (a) $a_{\gamma, -1}$ is the principal wave invariant
(\ref{PRIN});

\noindent(b) $L^{\#}$ is the primitive length of $\gamma$;
$\sigma$ is its Morse index; $P_{\gamma}$ is its Poincar\'e map;

\noindent(c) $B_{\gamma 0; j}$ has the form:
$$B_{\gamma 0;j} = \frac{1}{L^{\#}} \int_o^{L^{\#}} [a\; |\dot{Y}|^4 +
 b_1\; \tau |\dot{Y}\cdot Y|^2 + b_2\; \tau {\rm Re\,} (\bar{Y} \dot{Y})^2
+c \;\tau^2 |Y|^4 + d \;\tau_{\nu \nu}|Y|^4 + e\; \delta_{j0} \tau
]ds  $$
$$ +\frac{1}{L^{\#}}\sum_{0\leq m,n \leq 3; m+n=3}
C_{1;mn}\;\frac{\sin ((n-m)\alpha)}{|(1-e^{i(m-n)\alpha})|^2}\;
\left |\int_o^{L^{\#}}  \tau_{\nu}(s)\bar{Y}^m\cdot Y^n(s)ds
\right |^2 $$
$$+\frac{1}{L^{\#}}\sum_{0\leq m,n \leq 3; m+n=3} C_{2;mn}\;{\rm Im\,}
\; \left\{\int_o^{L^{\#}} \tau_{\nu}(s)\bar {Y}^m\cdot Y^n(s)
\left[\int_o^s\tau_{\nu}(t)\bar {Y}^n\cdot Y^m(t)dt\right] ds
\right\}$$ for various universal  coefficients.
 The coefficients $a, b_1, b_2, c, d, C_{j; mn}$ are universal and can be determined from
computable special cases. A somewhat computable case is that where $\gamma$ is a rotationally
invariant geodesic of length $L$ of a surface of revolution, in which case $\tau, \tau_{\nu}, \tau_{\nu \nu}$ are constants
and $y = e^{\frac{2 \pi i s}{L}}$ (we assume here that $\tau \equiv 1$ on $\gamma$).  

If we iterate $\gamma$ (i.e. consider $\gamma^k$, traversing $\gamma$ k times), the Floquet invariants become
independent and the two $B_{\gamma 0; j} (j = 0, 4$) are independent. Thus each is a spectral invariant. This
is a general phenomenon \cite{G,G2,Z3}. The coefficients $B_{\gamma 0; j}$ are Birkhoff canonical form invariants.

\subsection{Bounded smooth plane domains}

Euclidean plane domains $\Omega$ are quite a different problem because the metric is flat (and known) and it is
the boundary which is unknown. We express it locally as a graph $y = f(x)$ over the $x$-axis.  The role
of closed geodesics is played by periodic trajectories of the biliiard flow on $\Omega$. In particular we consider
`two-bounce' or bouncing ball orbits, which are straight line segments hitting the boundary orthogonally at
both endpoints. They always exist if the domain has an isometric involution or if it is convex. 

Rather than giving a formula for the subprincipal wave invariant, we give the top order terms (in derivatives
of the boundary) for every wave invariant. 
Modulo terms involving $\leq 2j -
2$ derivatives, the wave trace 
invariants 
 take the form
\begin{equation}  \label{BGAMMAJSYMpre}  \begin{array}{lll} B_{\gamma^r, j -
 1} & = &    (4 L r) {\mathcal A}_r(0) i^{j - 1} \{ 2 (w_{{\mathcal G}_{1, j}^{2j, 0}})\;
 (h_{2r}^{11})^j f^{(2j)}(0) \\ && \\ && + 4 (w_{{\mathcal G}_{2, j + 1 }^{2j - 1, 3,
 0}}) \;
 (h_{2r}^{11})^j \frac{1}{2 - 2 \cos \alpha/2} ( f^{(3)}(0) f^{(2j - 1)}(0)) \\ && \\
 && +  4 (w_{\widehat{{\mathcal G}}_{2, j + 1 }^{2j - 1, 3, 0}})\; (h_{2r}^{11})^{j - 2}
 \sum_{q = 1}^{2r} (h_{2r}^{1 q})^3 ( f^{(3)}(0) f^{(2j - 1)}(0))\}.\end{array}  \end{equation}

 Here 
 \begin{itemize}

\item $h_{2r}^{pq}$ are the matrix elements of the inverse of the
Hessian $H_{2r}$ of the length function ${\mathcal L}$ in suitable
(Cartesian graph) coordinates at $\gamma^r$;

\item ${\mathcal A}_r(0)$ is an $\Omega$-independent  (non-zero)
 constant obtained from amplitude of  the principal terms at
the critical bouncing ball orbit.

\item $w_{{\mathcal G}_{1, j}^{2j, 0}}$ (etc.) are certain
non-zero combinatorial constants  associated to certain  Feynman diagrams (graphs).

\end{itemize}

In this case, the Floquet exponents are buried in the Hessian matrix elements.  A higher dimensional
generalization of this formula is given in \cite{HZ2}.

\section{\label{IRST} Inverse results}

We now give an idea of how one can prove inverse spectral results using the wave trace invariants. 

\subsection{Plane domains with an up-down symmetry}

\subsubsection{Domains with one symmetry}

The class
  $\mathcal D_{1, L}$ consists of simply connected real-analytic plane domains $\Omega$ satisfying:
\begin{itemize}

\item  (i) There exists an  isometric involution $ \sigma$ of $\Omega$;

\item (ii) $\sigma$ `reverses' a  non-degenerate bouncing ball orbit $  \gamma \to \gamma^{-1}$;

\item (iii) The lengths $2 r L$ of  all iterates $\gamma^r$ have  multiplicity one in $Lsp(\Omega)$,  and the eigenvalues of the linear
Poincare map $P_{\gamma}$ are  not
roots of unity;

\end{itemize}

Let Spec$(\Omega)$ denote the spectrum of the Laplacian $\Delta_{\Omega}$  of the
domain $\Omega$ with Dirichlet boundary conditions.

\begin{theo} \label{ONESYM1D}
Spec$: {\mathcal D}_{1, L} \mapsto \R_+^{{\bf N}}$ is 1-1. \end{theo}

 Here, $Lsp(\Omega)$ denotes  the length spectrum, i.e. the set of lengths of closed trajectories of the
billiard flow of $\Omega$ (see
\S 2 for backgroun on billiards).
By a bouncing ball orbit $\gamma$ is meant a 2-link periodic trajectory of the billiard flow. It corresponds
to  a line segment in the interior of $\Omega$
 which intersects $\partial \Omega$ orthogonally at both boundary points.  By rotating
and translating $\Omega$ we may assume that $\gamma$ is vertical and that  $A = (0,-\frac{L}{2})$.
In a strip $T_{\epsilon}(\overline{AB})$, we may locally express
 $\partial \Omega = \partial \Omega^+ \cup \partial \Omega^-$ as the union of two graphs
over the $x$-axis, namely
\begin{equation}\label{GRAPHS}  \partial \Omega^+ = \{ y = f_+(x),\;\;\; x \in (-\epsilon, \epsilon )\},\;\; \partial \Omega^- = \{ y = f_-(x),\;\;\; x \in
(-\epsilon, \epsilon) \}. \end{equation}
The 
symmetry assumptions (i) - (ii)  is that
 $f_+(x) =   - f_-(x).$  
 Condition (iii) on the  multiplicity of $2 L$  means that just the two orbits  $\gamma, \gamma^{-1}$ have
length $2L$.   In the elliptic case, its eigenvalues $\{e^{\pm i \alpha}\}$ are
of modulus one and we are requiring that $\frac{\alpha}{2 \pi} \notin \Q.$ In the hyperbolic case, its eigenvalues
$\{e^{\pm \lambda}\}$ are real and they are never roots of unity in the non-degenerate case. These are generic
conditions in the class of analytic domains.

\begin{cor} \label{ONESYMCOR} Let ${\mathcal D}$ be the class of  real analytic plane domains with an isometric involution $\sigma$
satisfying:

\begin{itemize}

\item (i) $\sigma$ reverses the shortest closed billiard trajectory $\gamma$;

\item (ii) $r L_{\gamma}$ are of multiplicity one in $Lsp(\Omega).$

\end{itemize}

 Then Spec: ${\mathcal D} \mapsto \R_+^{{\bf N}}$ is 1-1. \end{cor}

The proof of the corollary from the Theorem is just to observe that any  shortest closed trajectory is automatically
a bouncing ball orbit. Its length is a spectral
invariant.

\subsubsection{Sketch of the proof}

 The
key point is determine the $2j-1$st and $2j$th Taylor coefficients
of the curvature at each reflection point from the $j-1$st wave
trace invariant for $\gamma$ and its iterates $\gamma^r.$
If the domain has the symmetries of an ellipse, 
 one obtains the Taylor coefficients immediately from (\ref{BGAMMAJSYMpre}), since the odd Taylor
 coefficients are zero. On the other hand, there is an obstruction
 to recovering the Taylor coefficients of $f$ when there is only
 one symmetry: namely, we must recover two Taylor coefficients
 $f^{(2j)}(0), f^{(2j - 1)}(0)$ for each new value of $j$ (the
 degree of the singularity). This is the principal obstacle to
 overcome.

 The expression
(\ref{BGAMMAJSYMpre}) for the Balian-Bloch invariants of $\gamma,
\gamma^2, \dots$ consists of two types of  terms, in terms of
their dependence on the iterate $r$. They have a common factor of
$2 r L (h_{2r}^{11})^{j - 2} {\mathcal A}_r(0) $, and after
factoring it out we obtain one term
$$ (h_{2r}^{11})^2 \{
(w_{{\mathcal G}_{1, j}^{2j, 0}}) f^{(2j)}(0) + \frac
{(w_{{\mathcal G}_{2, j + 1 }^{2j - 1, 3,
 0}})}{2 - 2 \cos
\alpha/2} f^{(3)}(0) f^{(2j - 1)}(0)\}  $$
 which  depends on the iterate $r$ through the coefficient
  $ (h_{2r}^{11})^2$, and one $$ (w_{\widehat{{\mathcal G}}_{2, j + 1 }^{2j - 1, 3, 0}}) \left(
\sum_{q = 1}^{2r} (h_{2r}^{1 q})^3 f^{(3)}(0) f^{(2j - 1)}(0)
\right)$$
 which depends on $r$ through the cubic sums $ \sum_{q = 1}^{2r}
(h_{2r}^{1 q})^3$ of inverse Hessian matrix elements
$h_{2r}^{pq}$. In order to `decouple' the even and odd
derivatives, it suffices to show that the functions
$(h_{2r}^{11})^2 $ and $\sum_{q = 1}^{2r} (h_{2r}^{1 q})^3$ are,
at least for `most' Floquet angles $\alpha$,  linearly independent
as functions of $ r\in \Z$, i.e. that $(h_{2r}^{11})^{-2} \sum_{q
= 1}^{2r} (h_{2r}^{1 q})^3$ is a non-constant function of $r$. It
is convenient to use the parameter $a = - 2 \cos \frac{\alpha}{2}$
and we write the dependence as $h_{2r}^{ij}(a)$.

 We therefore define the `bad' set of Floquet angles by
\begin{equation} \label{BAD} {\mathcal B} = \{a:\; \;\mbox{the sequence }\; \{(h_{2r}^{11}(a))^{-2}
\sum_{q = 1}^{2r} (h_{2r}^{1 q}(a))^3, \;\; r = 1, 2, 3, \dots \}
\; \mbox{is constant in }\; r\}. \end{equation} Using facts about
the finite Fourier transform and circulant matrices, we compute
that $ {\mathcal B} = \{0, 1, \pm 2\}.$  For Floquet angles outside of ${\mathcal B}$,
  we can determine
 all Taylor coefficients $f_+^{(j)}(0)$ from the wave
 invariants and hence the analytic domain.

\subsection{A higher dimensional inverse result}

It turns out that one can generalize the result to higher dimensions \cite{HZ2} as long as the domain
has $\Z_2$ symmetries along all coordinate axes. The  $(\Z/2\,\Z)^n$ symmetries   are the maps
\begin{equation} \label{SIGMAJ} \sigma_j: (x_1, \dots, x_n) \to (x_1, \dots, - x_j, x_{j +
1}, \dots, x_n). \end{equation}  The symmetry assumption implies that the
intersections of the coordinate axes with $\Omega$ are projections
of  bouncing ball orbits preserved by the symmetries.

We denote by $\mathcal D_{ L}$  the class of all bounded
real-analytic  domains  such
that:

\begin{equation} \label{DL}
\left\{\begin{array}{llllll}
   \text{(i)}\; \; \;\sigma_j: \Omega \to \Omega \; \text{is an isometry for all}\; j=1,\dots, n; \\

   \text{(ii)} \;\text{one of the coordinate axis bouncing ball orbits, called $\gamma$, is  of length}\; 2L \;
             \;  \;\\

   \text{(iii)} \;\text{the lengths}\; 2 r L \; \text{of  all iterates}\; \gamma^r (r =
             1, 2, 3, \dots) \; \text{have multiplicity one in}\; Lsp(\Omega); \\

   \text{(iv)}\; \gamma \;\;  \text{is non-degenerate, i.e. 1 is not an eigenvalue of  its Poincar\'e map }  P_{\gamma}; \\
     \;\;\;\;\;\;\;\text{if}\; \gamma \;\text{is elliptic and}\; \{e^{\pm i\al_1},.
   ..e^{\pm i \al_{n-1}}\}\; \text{are the eigenvalues of } \; P_{\gamma}\; , \text{we} \\
              \; \; \; \;\; \;\;\text {further require that}\; \{\al_1,..., \al_{n-1}\} \;
              \text{are linearly independent over}\; \mathbb Q. \; \text{We assume the }\\ \;\;\;\;\;
              \;\;
              \text{same independence  condition in the Hyperbolic case or mixed cases.}\\
\end{array} \right.
\end{equation}

Multiplicity one means that there exists precisely one closed
billiard trajectory of the given length up to time reversal. Let
Spec$_B(\Omega)$ denote the spectrum of the Laplacian
$\Delta_{\Omega}^B$  of the domain $\Omega$ with boundary
conditions $B$ (Dirichlet or Neumann).

\begin{theo} \label{ONESYM} \cite{HZ2} For Dirichlet (or Neumann)  boundary
conditions $B$, the map Spec$_B: {\mathcal D}_{L} \to
\R_+^{{\bf N}}$ is 1-1. \end{theo}

In other words, if two bounded real analytic domains $\Omega_1,
\Omega_2 \subset \R^n$ possessing the symmetries of an ellipsoid
and satisfying the non-degeneracy and length  assumptions of
(\ref{DL}) have the same Dirichlet (resp. Neumann) spectra, then
they are isometric. To our knowledge, the  only prior positive
result on the inverse spectral problem for higher dimensional
bounded domains is that a domain with the Dirichlet (or Neumann)
spectrum of a ball must be a ball.

The proof of Theorem \ref{ONESYM} is similar in nature to the proof of Theorem \ref{ONESYM1D} but
uses the analysis in  \cite{He} of the  much more complicated dependence on Floquet eigenvalues in higher dimensions.

\subsection{\label{SURFREVCALC}Spectral determination of simple surfaces of revolution}

A result which is somewhat parallel to Theorem \ref{ONESYM} is that convex analytic surfaces of revolution
(satisfying a non-degeneracy hypothesis)  are determined by their spectra (in this class). The
 profile curve of the surface of revolution is analogous to the `top half' of the boundary of the plane domain,
and the $S^1$ symmetry is analogous to the symmetry of the top and bottom halves.

The precise class of  metrics we consider are those  metrics $g$
on $S^2$ which belong to the
 class $\rcal^*$ of real analytic, rotationally
invariant metrics on $S^2$ with simple length spectrum in the
above sense and satisfying the following `simplicity' condition
$$\begin{array}{ll} \bullet & g = dr^2 + a(r)^2 d\theta^2 \\ & \\
\bullet & \exists ! r_0 : a'(r_0) = 0; \\ & \\
\bullet & \mbox{The Poincare map}\;\; \pcal_0\;\; \mbox{of}\;\;
r = r_0 \;\;\mbox{is elliptic of twist type} \end{array} $$

Convex analytic surfaces of revolution are examples, but
there are others. The unique isolated closed geodesic (at distance
$r_0$) is an elliptic orbit.

\begin{theo}\label{ISPSR} (\cite{Z3}) Suppose that $g_1, g_2$ are two real analytic metrics in the class $ \rcal^*$ on
$S^2$, i.e.  $(S^2, g_i)$ are simple surfaces of revolution
for which the meridian geodesic length is simple.   Then Sp($\Delta_{g_1}$) =
Sp($\Delta_{g_2}$) implies $g_1 = g_2.$ \end{theo}

We discuss the proof in \S \ref{SURFREV}. We note that the same kind of result should hold for Schr\"odinger
operators \eqref{SO} for fixed  $\hbar = 1$, where one metric $g \in \rcal^*$ is fixed and an $S^1$-invariant
potential $V$ is varied. However, the `simplicity' assumption probably  needs to be placed on $V$ and is probably
similar to the monotonicity assumption in \cite{DHV}.

\subsection{\label{ISODEF} Isospectral deformations and spectral rigidity}

As mentioned in the introduction, one may at least heuristically consider the map $Sp$ on the space of
metrics, potentials, or domains. The derivative of $Sp$ is the 
map $\delta g \to (\delta \lambda_1, \delta \lambda_2, \dots) $  from the tangent space to metrics, potentials or domains to sequenes. One may consider variations of the spectral functions, wave invariants etc. Some early works on variations
of eigenvalues are \cite{Be,BE,U}.

If $\gamma$ is an isolated, non-degenerate closed geodesic of $g$,
then for any deformation $g_t$ of  $g$,  $\gamma$ deforms smoothly
as a closed geodesic $\gamma_t$ of $g_t$ and one may define its
variation
\begin{equation} \dot{L_{\gamma}} = \frac{d}{dt}|_{t = 0} L_{\gamma_t}. \end{equation}
It is not hard to compute that \begin{equation} \label{DOTL}
\dot{L_{\gamma}} =  \int_{\gamma} \dot{g} ds,
\end{equation}
where $\dot{g}$ is viewed as a quadratic function on $TM$ and
$\gamma$ is viewed as the curve $(\gamma(t), \gamma'(t))$ in $TM$.

It follows that whenever the closed geodesics are non-degenerate
and of multiplicity one in $Lsp(M, g)$, an isospectral deformation preserves lengths and so
\begin{equation}\label{GDOT}  \int_{\gamma} \dot{g} ds = 0, \;\;\; \forall
\gamma.
\end{equation}
In \cite{GK} and in many subsequent articles, it is shown that there cannot exist such a deformation 
among negatively curved surfaces. The proof is that there does not exist a (non-trivial) symmetric tensor
$\dot{g}$ satisfying \eqref{GDOT}. To our knowledge the analogous result for domains with boundary is unknown.
The analogue of a negatively curved surface is a hyperbolic billiard table. Hyperbolicity refers to the Anosov property
of the billiard flow. Such a table could be the complement of a disc (or more than one disc) in a negatively curved
surface, or a plane domain whose boundary consists of a finite number of concave curves meeting at corners. 

\begin{mainprob} Can one generalize the result of \cite{GK} to hyperbolic billiards?  Does a hyperbolic billiard table admit isospectral deformations? \end{mainprob}

\section{Geometric calculation of spectral invariants}

In this section, we briefly explain how one calculates the wave trace invariants explicitly in terms
of geometric invariants or (in the boundary case) in terms of the boundary defining function.

\subsection{Parametrix `path integral' versus canonical Hamiltonian approach}

Roughly speaking, one has two approaches to constructing the `propagator' $U(t) = e^{ - i t \sqrt{\Delta_g}}$
in quantum mechanics and studying its trace. One is the `Hamiltonian approach' or `Birkhoff canonical forms' approach in which one constructs
the propagator locally around a closed geodesic $\gamma$ by adapting the canonical formalism (creation/annihilation operators,
harmonic oscillators, etc.) to a tubular neighborhood of $\gamma$. The canonical formalism is essentially the
representation theory of the Heisenberg algebra and the symplectic algebra on $L^2(\R^n)$. It is not obvious that
one can `transplant' the canonical formalism onto a Riemannian manifold, but this is what the theory of quantum
Birkhoff normal forms achieves. Its origins lie in constructions of Gaussian beams by Babich, Lazutkin, Ralston
and others (see \cite{BB}).  Although Gaussian beams only exist for  stable elliptic closed geodesics,  the Birkhoff canonical
form construction works locally around any non-degenerate closed geodesic (and even around degenerate ones).

The second approach is the  Lagrangian (or `path integral') approach in which one constructs a parametrix for the wave group. Unlike the canonical
forms approach, there is nothing canonical about the parametrix construction. It is an art to construct a parametrix
well adapted to a given problem. The parametrix approach is the  traditional one and the reader is probably familiar with the heat kernel parametrix (if not  with the
term `parametrix').

Let us review the original wave kernel parametrix for \eqref{Et} due to J. Hadamard (1921). 
On a Riemannian manifold without boundary,  one can construct  a  Hadamard
parametrix for $\cos t \sqrt{\Delta} (x,x)$ for  small times, i.e. an approximation for the Schwartz kernel
$E(t, x, y)$  of the solution operator
of the Cauchy problem,
\begin{equation} \left\{ \begin{array}{ll} (\frac{\partial}{\partial t}^2 - \Delta ) u = 0& \\
u|_{t=0} = f & \frac{\partial}{\partial t} u |_{t=0} = 0.
\end{array}\right.,\end{equation}
The parametrix is an oscillatory integral
\begin{equation}  E(t, x, y) \sim \int_{0}^{\infty} e^{i \theta (r^2-t^2)}
\sum_{j=0}^{\infty} U_j(x,y) \theta_{\mbox{reg}}^{\frac{n-1}{2} - j}
d\theta
 \;\;\;\mbox{mod}\;\;C^{\infty}  \end{equation}
where the Hadamard-Riesz coefficients $U_j$ are  determined
inductively by the transport equations
\begin{equation}\begin{array}{l}
 \frac{\Theta'}{2 \Theta} U_0 + \frac{\partial U_0}{\partial r} = 0\\ \\
4 i r(x,y) \{(\frac{k+1}{r(x,y)} +  \frac{\Theta'}{2 \Theta})
U_{k+1} + \frac{\partial U_{k + 1}}{\partial r}\} = \Delta_y WU_k.
\end{array}\end{equation} The solutions are given by:
\begin{equation}\label{HR} \begin{array}{l} U_0(x,y) = \Theta^{-1/2}(x,y) \\ \\
U_{j+1}(x,y) =  \Theta^{-1/2}(x,y) \int_0^1 s^j \Theta(x,
x_s)^{1/2} \Delta_2  U_j(x, x_s) ds
\end{array} \end{equation}
where $x_s$ is the geodesic from $x$ to $y$ parametrized
proportionately to arc-length and where $\Delta_2$ operates in the
second variable. Also, $r^2(x, y)$ is the distance squared,  $dV_g = \Theta(x, y) dy$   in normal coordinates centered at $x$, and $\theta_{\mbox{reg}}^{\frac{n-1}{2} - j}$ refers to the regularization of this distribution by analytic continuation (M. Riesz).

The parametrix is a `geometric construction' of the wave group and one may determine the singularities of the wave
trace by expressing it in terms of the parametrix. The Hadamard parametrix above is useful to obtain the singularity
at $t = 0$ and contains the information of the heat trace there. 
But the above  parametrix is only valid for small $t$ since it uses $r^2(x, y)$ and $\Theta(x, y)$. To obtain information 
at much longer times such as $t \in Lsp(M,g)$ one needs to construct a parametrix which is valid at that time
around a closed geodesic. There are a number of clever ways to do this but none are simple except in special
cases such as manifolds with no conjugate points. In that case one may follow the method of Selberg (1956)
to construct the wave kernel on the universal cover and summing over the deck transformation group (see
\cite{Z1} for references and background).

\subsection{\label{BB} Domains with boundary}

In the case of bounded domains, it is difficult to construct a parametrix that is valid in the glancing region
of geodesics tangent to the boundary.  A microlocal parametrix for the Dirichlet or Neumann
even wave group  $E_B^{\Omega}(t) = \cos t \sqrt{\Delta_B}$ near a
transversal reflecting ray  was constructed by J. Chazarain
\cite{Ch} (see also \cite{GM, PS} for more details ).

To prove Theorem \ref{ONESYM1D}, the author  used the construction of the wave group (or rather, Green's function)
using  the method of layer
potentials as in Balian-Bloch \cite{BB1,BB2}  rather than using a parametrix. This method is used universally
in the physics literature to make calculations of eigenvalues and eigenfunctions and it is natural to use it
for inverse spectral problems.

We fix
 a non-degenerate bouncing ball orbit  $\gamma$  of length
 $L_{\gamma}$, and
  let  $\hat{\rho} \in C_0^{\infty}(L_{\gamma} - \epsilon, L_{\gamma} + \epsilon)$ be a cutoff,  equal to one on an interval  $(L_{\gamma} - \epsilon/2, L_{\gamma} + \epsilon/2)$ which contains
no other lengths in Lsp$(\Omega)$ occur in its support. The wave trace invariants associated to $\gamma$ are
the coefficients of the asymptotic expansion as $k \to \infty$ of 
\begin{equation} \label{RW} R_{\rho B}^{\Omega}(k + i \tau)
= \int_0^{\infty} \hat{\rho}(t) e^{i (k + i \tau) t}
 E_{B}^{\Omega}(t) dt. \end{equation}
When $\gamma, \gamma^{-1}$ are the unique closed orbits of length
$ L_{\gamma}$,  it follows from the Poisson relation for manifolds
with boundary (see \S 3 and \cite{GM}) that \eqref{RW} has the expansion
\begin{equation} \label{PR} Tr 1_{\Omega} R_{\rho}(k + i \tau) \sim  e^{ (i k - \tau) L_{\gamma}}  \sum_{j = 1}^{\infty} (B_{\gamma, j}
+  B_{\gamma^{-1}, j}) k^{-j},\;\;\; k \to \infty \end{equation}
with coefficients $B_{\gamma, j},  B_{\gamma^{-1}, j}$ determined
by the jet of $\Omega$ at the reflection points of $\gamma$. The
coefficients $B_{\gamma, j},  B_{\gamma^{-1}, j}$ are thus
essentially the same as the wave trace
 coefficients at the singularity  $t = L_{\gamma}$.

As mentioned above, it is the art in  inverse spectral theory to find a method by which the coefficients
can be calculated explicitly. One might use a well-chosen parametrix around $\gamma$ or use a canonical form.
In \cite{Z4}, we use the layer potential method of \cite{BB1,BB2} to  express the Dirichlet (resp. Neumann)
resolvent in terms of the `free resolvent'  $R_0(k + i \tau) = - (\Delta_0 + (k +
i \tau)^2)^{-1}$  on $\R^2$, in the form,
\begin{equation} \label{POT}  R_{\Omega}(k + i \tau) =
 R_0(k + i \tau) - {\mathcal D} \ell(k + i \tau) (I +  N(k + i \tau))^{-1}
 \gamma
 {\mathcal S} \ell^{tr}(k + i \tau). \end{equation}
Here,   $\gamma: H^{s}(\Omega) \to
H^{s-1/2}(\partial \Omega)$ is the restriction to the boundary,
and
 ${\mathcal D} \ell (k + i \tau)$ (resp. ${\mathcal S} \ell (k + i \tau)$) is the double
 (resp. single) layer potential, i.e.   the operator from $H^s(\partial \Omega) \to H^{s+ 1/2}_{loc}
(\Omega)$ defined by
\begin{equation} \label{layers}\left\{ \begin{array}{l} {\mathcal S} \ell(k + i \tau)  f(x) = \int_{\partial \Omega} G_0(k + i \tau, x, q) f(q)
d s (q), \\ \\   {\mathcal D} \ell(k + i \tau)f(x) =
\int_{\partial \Omega}\frac{\partial}{\partial \nu_y} G_0(k + i
\tau, x, q) f(q) ds(q),
 \end{array} \right. \end{equation}
where $ds(q)$ is the arc-length measure on $\partial \Omega$,
where $\nu$ is the interior unit normal to $\Omega$. Also, $\SL^{tr}$ is its
transpose (from the interior to the boundary), and $G_0(z, x, y)$
is the free Green's function, i.e. the  kernel of $R_0(z)$.
Further,
\begin{equation} \label{blayers} N(k + i \tau)f(q) =  2 \int_{\partial
\Omega}\frac{\partial}{\partial \nu_y} G_0(k + i \tau, q, q')
f(q') ds(q')
\end{equation} is the boundary integral operator induced by ${\mathcal D} \ell$. 
One may simplify the trace by combining the interior and exterior problems.  We write $L^2(\R^2) = L^2(\Omega)
\oplus L^2(\Omega^c)$
 and  let $R_{N }^{\Omega},$ resp. $ (k + i \tau), R_{ D}^{\Omega^c}(k + i \tau)$ denote the Neumann resolvent
 on the exterior domain, resp. the Dirichlet resolvent on the interior domain.  We then regard
 $R_{ D}^{\Omega^c}(k + i \tau)
 \oplus  R_{N }^{\Omega} (k + i \tau) $ as an operator on this
 space. We cycle around the layer potentials in (\ref{POT}) when
 taking the trace to obtain,

\begin{equation} \label{IO}
\begin{array}{l}
 Tr_{\R^2} [R_{ D}^{\Omega^c}(k + i \tau)
 \oplus  R_{N }^{\Omega} (k + i \tau)  - R_{0 }(k + i \tau)] \;= \;  \frac{d}{d k } \log \det
 \bigg(I + N(k + i \tau)\bigg)
,\end{array}
\end{equation}
where the determinant is the usual Fredholm determinant. The main point then is to prove the existence
of the following asymptotic expansion:

\begin{prop} \label{MAINCOR}  Suppose that $L_{\gamma}$ is the only length
in the support of $\hat{\rho}$. Then,
$$\int_{\R} \rho(k -
\lambda) \frac{d}{d \lambda} \log \det (I + N(\lambda + i \tau)) d
\lambda  \sim \sum_{j = 0}^{\infty} B_{\gamma; j} k^{-j}, $$ where
$B_{\gamma; j}$ are the wave invariants of $\gamma$.
\end{prop}

One then shows that $N(k + i\tau)$ is a semi-classical Fourier integral operator quantizing the billiard map
on the boundary. Let us explain these terms.
The billiard map $\beta : B^* \partial \Omega \to B^* \partial \Omega$ on the ball bundle of $\partial \Omega$  is defined as follows: given $(y, \eta) \in B^* \partial \Omega$, with $|\eta| < 1$ we let $(y, \zeta) \in S^* \Omega$ be the unique inward-pointing unit covector at $y$ which projects to  $(y, \eta)$ under the map $T^*_{\partial \Omega} \Omega \to T^* \partial \Omega$. Then we follow the geodesic (straight line) determined by $(y, \zeta)$ to the first place it intersects the boundary again; let $y' \in Y$ denote this first intersection, and   let $\eta'$ be the projection of $\zeta$ to $B^*_{y'} \partial \Omega$. Then we define
$$
\beta(y, \eta) = (y', \eta').
$$
We say that $N(k + i \tau)$ quantizes the billiard map in that its Schwartz kernel 
is (almost) given by 
$$N(k + i \tau, y, y') \sim
 (k + i \tau)^{(n-2)} e^{i (k + i \tau) |y-y'|} \tilde b( k |y-y'|),
$$
where $\tilde b(t)$ has an expansion in inverse powers of $t$ as $t \to \infty$, with leading term $\sim t^{(- n+1)/2}$. This is
proved in \cite{HaZe,Z4}. The kernel is a semiclassical FIO of order $0$ as long as $|y -y'| > k^{-\half} \log k$.   We separate out the tangential and
transversal parts of $N$ by introducing  a cutoff of the form
$\chi(k^{1 - \delta} |q - q'| )$ to the diagonal, where $\delta >
1/2$ and where $\chi \in C_0^{\infty}(\R)$ is a cutoff to a
neighborhood of $0$. We then put
\begin{equation} \label{N01}  N(k + i \tau) = N_0(k + i \tau) + N_1(k + i
\tau), \;\; \mbox{with}  \end{equation}
\begin{equation} \label{N01DEF} \left\{\begin{array}{l}  N_0(k + i
\tau, q, q') = \chi(k^{1 - \delta} |q - q'| ) \;  N(k + i \tau, q,
q'), \\ \\ N_1(k + i \tau, q, q') = (1 - \chi(k^{1 - \delta} |q -
q'| ))\; N(k + i \tau, q, q'). \end{array} \right.
\end{equation}
The term $N_1$ is a semiclassical Fourier integral kernel
quantizing the billiard map, while $N_0$ behaves like an Airy
operator close to the diagonal with the singularity of a
homogeneous pseudodifferential operator on the diagonal.

The advantage of this approach over the canonical forms method or using a Chazarain (or other parametrix)
is that the kernel $N(k, q, q')$ is independent of the boundary except for the restriction of the layer potential
to the given boundary. In some sense the expansion above is `canonical'. The
phase
$|y - y'|$ on $\partial \Omega \times \partial \Omega$ is the generating function of the billiard map as long 
as $\Omega$ is  convex in the sense of symplectic geometry. If $\Omega$ fails to be convex, it generates a canonical
relation which is larger than the graph of $\beta$, causing some technical problems that can be handled \cite{HaZe}. 
One also must take into account the part of the kernel very close to the diagonal, which is particularly important
if one is studying billiard trajectories that creep along the boundary but it is not important if one considers
trajectories that are transverse to the boundary such as bouncing ball orbits.

At least formally, we then
 expand  $(I +  N(k + i \tau))^{-1}$ in a finite geometric
series plus remainder:
\begin{equation} \label{GS} (I \!+\! N(\lambda \!+\! i \tau))^{-1} = \sum_{M = 0}^{M_0} (-1)^M \; N(\lambda)^M
+
 (-1)^{M_0 + 1} \; N(\lambda)^{M_0 + 1}  (I
\!+\! N(\lambda \!+\! i \tau))^{-1}. \end{equation} 
We then write
\begin{equation} \label{BINOMIAL} (N_0 + N_1)^M = \sum_{\sigma: \{ 1, \dots, M\} \to \{0, 1\}}
N_{\sigma(1)} \circ N_{\sigma(2)} \circ \cdots \circ
N_{\sigma(M)}. \end{equation} We  regularize $N^M$ by eliminating
the factors of $N_0$ from each of these terms. This is obviously
not possible for the term $N_0^M$ but  it is possible for the
other terms. By explicitly writing out the composition in terms of
Hankel functions and using the basic identities for these special
functions, we prove that $N_0 \circ N_1 \circ \chi_0 (k + i \tau,
\phi_{1}, \phi_{2}) )$ is a semiclassical Fourier integral
operator on $\partial \Omega$ of order $-1$ associated to the
billiard map. Thus,  composition with $N_0$ lowers the order. The
remaining terms $N_0^M$, when composed with a cut-off to $\gamma$,
do not contribute asymptotically to the trace.

The successive removal of  the factors of $N_0$ thus gives a
semi-classical quantization of the billiard map near $\gamma$. We
then calculate the traces of each term by the stationary phase
method to obtain the expansion \eqref{PR}. The coefficients containing the top derivatives for
each power of $k$ are rather straightforward to calculate explicitly.

\subsection{\label{BIRK} Birkhoff canonical form} So far we have described parametrix methods for evaluating wave trace
coefficients.
We now describe the Birkhoff canonical form and the sense in which it is canonical.
Let $\gamma$ be a non-degenerate  closed geodesic on an
$n$-dimensional Riemannian manifold, and at first let us  assume
it to be elliptic. For each transversal quantum
number $q \in \Z^{n-1}$, there was an approximate eigenvalue of
the form
\begin{equation} \label{APPROXEIG} \lambda_{kq} \equiv r_{kq} +
\frac{p_1(q)}{r_{kq}} + \frac{p_2(q)}{r_{kq}^2} +
...\end{equation} where
$$r_{kq} = \frac{1}{L} (2 \pi k + \sum_{j=1}^n (q_j + \frac{1}{2}) \alpha_j)$$
where the coefficients are polynomials of specified degrees and
parities. We refer to \cite{BB} for a clear exposition and for
details.

A natural question is whether the wave invariants $a_{\gamma j}$
of (\ref{WEXP}) can be determined from the quasi-eigenvalues
(\ref{APPROXEIG}), i.e. from the polynomials $p_j(q)$.   The answer is `yes' and in effect 
(\ref{APPROXEIG}) is the canonical form. 

To put $\Delta$ into normal form, is first  to
conjugate it
  into a distinguished
maximal abelian algebra ${\mathcal A}$ of pseudodifferential
operators on a model space, the cylinder
 $S^1_L \times \R^n$, where $S^1_L$ is the circle of length $L$.
  The algebra is  generated by the tangential operator
$D_s:=\frac{\partial}{i \partial s}$ on $S^1_L$ together with the
transverse action operators. The nature of these action operators
depends on $\gamma$. When $\gamma$ is elliptic, the action
operators are harmonic oscillators
$$I_j=I_j(y,D_y) := \frac{1}{2} (D_{y_j}^2 + y_j^2),$$
while in the real  hyperbolic case they have the form
$$I_j = y_j D_{y_j} + D_{y_j} y_j.$$
When $\gamma$ is non-degenerate,  they involve some mixture of
these operators (and also complex hyperbolic actions)  according
to the spectral decomposition of $P_{\gamma}$.
 For notational simplicity we restrict to the elliptic case and
 put
$$H_{\alpha}:= \frac{1}{2}\sum_{k=1}^n \alpha_k I_k $$
where $e^{\pm i \alpha_k}$ are the eigenvalues of the Poincare map
$P_{\gamma}$.

To put $\Delta$ into normal form is to conjugate it to the model
space and algebra as a function of $D_s$ and the action operators.
The conjugation is only defined in a neighborhood of $\gamma$ in
$T^*M-0$, i.e. one constructs  a microlocally elliptic
  Fourier Integral operator $W$
from the conic neighborhood  of $\R ^+\gamma$ in $T^*N_{\gamma}-0$
to a conic neighborhood of $T^*_{+}S^1_L$ in $T^*(S^1_L \times
R^n)$ such that:

\begin{equation} \label{QBNF} W \sqrt{\Delta_{\psi}}W^{-1}
 \equiv {\mathcal D}  +
\frac{\tilde{p}_1(\hat{I}_1,\dots,\hat{I}_n)}{L {\mathcal D}} +
 \frac{\tilde{p}_2(\hat{I}_1, \dots,\hat{I}_n)}{(L {\mathcal D})^2}
+\dots+\frac{\tilde{p}_{k+1}(\hat{I}_1,\dots,\hat{I}_n)}{(L
{\mathcal D})^{k+1}}+ \dots \end{equation} where the numerators
$p_j(\hat{I}_1,...,\hat{I}_n),
\tilde{p}_j(\hat{I}_1,...,\hat{I}_n)$ are polynomials of degree
j+1 in the variables $\hat{I}_1,...,\hat{I}_n$,  where $W^{-1}$
denotes a microlocal inverse to $W$. Here, ${\mathcal D} =  D_s +
\frac{1}{L}H_{\alpha}.$ The kth remainder term lies in the space
$\oplus_{j=o}^{k+2} O_{2(k+2-j)}\Psi^{1-j}$, where $\Psi^s$
denotes the pseudo-differential operators on the model space of
order $s$ and where  $O_{2(k+2-j)}$ denotes the operators whose
symbols vanish to the order $2(k + 2 - j)$ along $\gamma$. We
observe that (\ref{QBNF}) is an operator version of
(\ref{APPROXEIG}).
\bigskip

The inverse result of \cite{G}, see also \cite{Z3, Z4}] is:
\bigskip

\noindent{\bf Theorem  } {\it Let $\gamma$ be a non-degenerate
 closed geodesic. Then the quantum Birkhoff normal form
around $\gamma$ is a spectral invariant; in particular the
classical Birhoff normal form is a spectral invariant.}
\bigskip

In other words, one can determine the polynomials
$p_j(\hat{I}_1,...,\hat{I}_n)$ from the wave trace invariants of
$\Delta$ at $\gamma.$ The key point is that one can construct the normal form explicity by re-scaling
the metric around $\gamma$ and making a step-by-step conjugating to Hermite operators and creation/annihilation
operators first define by Babich, Lazutkin and others (see \cite{BB} for background).  Thus, the algorithm is quite
concrete and leads to computable spectra invariants of Theorem \ref{BNFDATA}.  But to date, the formulae have not been applied
to any concrete inverse spectral problems due to their complexity. In \cite{Z2, HZ2} we found that inverse results could
be proved in the boundary case  by just keeping track of the highest order derivative terms in the wave invariants. It is
possible that the argument could be generalized to some metric cases.


\subsection{\label{SURFREV} Proof of Theorem \ref{ISPSR} on spectral determination of analytic simple surfaces of revolution}

As mentioned above, surfaces of revolution are the surfaces most analogous to bounded plane domains in that both
are defined by a profile function of one real variable. This suggests that one should be able to prove an analogue
of Theorem \ref{ONESYM1D} for analytic surfaces of revolution. In fact, it was proved first and by a normal forms method.
To our knowledge, it remains the only inverse spectral result for metrics which was proved by using Birkhoff normal
forms, and therefore we briefly recall the method.  Recently, similar methods have been used in the $\hbar$  potential
problem for \eqref{SO} in several settings \cite{CdVG,CdV2,GW,GU}.  

The proof is based on quantum Birkhoff normal forms for the
Laplacian $\Delta$. But the special feature of simple surfaces of
revolution is that there exists a global Birkhoff normal form as
well as local ones around the critical closed orbit or the
invariant tori. It was constructed in \cite{CdV4} and expresses the fact tht $\Delta$ is a toric integrable
Laplacian in the following sense: there exist commuting first
order pseudo-differential operators $\hat{I}_1, \hat{i}_2$ such
that:
\begin{itemize}

\item the joint spectrum is integral, i.e.
 $Sp(\lcal) \subset \Z^2 \cap \Gamma + \{\mu\}$ where $\Gamma$ is the cone $I_2 \geq |I_1|$ in $\R^2$.

\item The square root of $\Delta$ is a first order polyhomogenous function  $\sqrt{\Delta} =
\hat{H}(\hat{I}_1, \hat{I_2})$ of the action operators.
\end{itemize}

By polyhomogeneous, we mean that $\hat{H}$ has an asymptotic
expansion in homogeneous functions of the form:
$$\hat{H} \sim H_1 +  H_{-1} + \dots, \;\;\;\;\;H_j( r I) = r^j
H_j(I).$$ The principal symbols $I_j$ of the $\hat{I}_j$'s
generate a classical Hamiltonian torus action on $T^*S^2 - 0$.
There is no term of order zero.
It follows that
$$Sp(\sqrt{\Delta_g}) = \{ \hat{H} ( N + \mu): N \in \Z^2 \cap \Gamma_o\},$$
where the eigenvalues have expansions
$$\lambda_N \sim H_1(N + \mu) + H_{-1}(N + \mu) + \dots.$$

\begin{theo}(\cite{Z2}) \label{SRNORM} Let $(S^2, g)$ be an analytic simple surface of revolution with simple
length spectrum.  Then the normal form $\hat{H}(\xi_1, \xi_2)$  is
a spectral invariant.\end{theo}

The normal form and the proof are very different from the
non-degenerate case in \cite{G, Z3, Z4}. Recently, similar kinds of results have been proved in other integrable
settings \cite{CPV,ChV}

To complete the proof of  Theorem \ref{ISPSR}, we need to show
that $\hat{H}$ determines a metric in $\rcal.$ As in the
bounded domain case outline above, we need to 
calculate the normal form invariants. It turns out to be
sufficient to calculate $H_1 = H$ and $H_{-1}$ in terms of the
metric (i.e. in terms of $a(r)$) and then to invert the
expressions to determine $a(r)$.
In  \cite{Z2}   $H$ and $H_{-1}$
were calculaed using by studying asymptotics $\sqrt{\Delta} =
\hat{H}(\hat{I}_1, \hat{I_2})$ along `rays of representations'  of
the quantum torus action, i.e. along multiples of a given lattice
point $(n_0, k_0).$    In effect, the $n_0$ parameter is like $\hbar^{-1}$
and we reduce to a one-dimensional Sturm-Liouville problem. Thus, Theorem \ref{SRNORM}
converts the hard problem into the easier problem of determining the potential
for  semi-classical Schr\"odinger operators. 
In a rather standard way, one finds that 
$$ \int_{\R} ( E - x)_+^{\half} d\mu(x)$$
is a spectral invariant, 
where $\mu$ is the distribution function $\mu (x):= |\{r:
\frac{1}{a(r)^2} \leq x\}|$ of   $\frac{1}{a^2}$, with $|\cdot|$
the Lebesgue measure. This  Abel transform is invertible and hence
$$d \mu(x) = \sum_{r: \frac{1}{a(r)^2} = x} |\frac{d}{dr} \frac{1}{a(r)^2} |^{-1}
dx$$ and therefore
$$ J(x) := \sum_{r: a(r) = x} \frac{1}{|a'(r)|}$$
are spectral invariants. By the simplicity assumption on $a$,
there are just two solutions of $a(r) = x$; the smaller will be
written $r_{-}(x)$ and the larger, $r_{+}(x).$ Thus,  the function
$$J(x) = \frac{1}{|a'(r_{-}(x))|} + \frac{1}{|a'(r_{+}(x))|} \leqno(6.9)$$
is a spectral invariant.

By studying  $H_{-1}$, we find in a somewhat similar way that
$$K(x) = |a'(r_{-}(x))| + |a'(r_{+}(x))|$$
is a spectral invariant. It follows that we can determine
$a'(r_{+\pm}(x))$.

Since both metrics $g_1$ and $g_2$ are assumed to belong to $\rcal^*$, they are determined by their respective functions $a_j(r)$.
We conclude that $a_1 = a_2$ and hence that $g_1 = g_2$.

In more recent work on the inverse problem for the $\hbar$-spectra of Schr\"odinger operators, the methods
are similar to the final step in the proof above but allow potentials which have more critical points. It is then
a challenge to recover the potential from its distribution function and the  higher invariants. Vice-versa, the
methods might be used for more general inverse problems for metrics on surfaces of revolution in which one
relaxes the simplicity assumption.

\section{Isospectral deformations of domains}

An isospectral deformation of a plane domain $\Omega_0$  is a
one-parameter family $\Omega_{\epsilon}$ of plane domains for
which the spectrum of the Euclidean Dirichlet (or Neumann)
Laplacian $\Delta_{\epsilon}$ is constant (including
multiplicities). We say that $\Omega_{\epsilon}$ is a $C^1$ curve
of $C^{\infty}$ plane domains if there exists a $C^1$ curve  of
diffeomorphisms $\phi_{\epsilon}$ of a neighborhood of $\Omega_0
\subset \R^2$  with $\phi_0 = id$ and with $ \Omega_{\epsilon} =
\phi_{\epsilon}(\Omega_0)$. The infinitesimal generator  $X =
\frac{d}{d \epsilon} \phi_{\epsilon}$ is a vector field in a
neighborhood of $\Omega_0$ which  restricts to a vector field
along  $\partial \Omega_0$;  we denote by $X_{\nu} = \dot{\rho} \nu$
its outer normal component. With no essential loss of generality we may
assume that $\phi_{\epsilon} |_{\partial \Omega_0}$ is a map of the
form \begin{equation} \label{rhodef}  x \in
\partial \Omega_0 \to x + \rho_{\epsilon} (x) \nu_x, \end{equation} where  $\rho_{\epsilon} \in
C^1([0, \epsilon_0], C^{\infty}(\partial \Omega_0))$, $\epsilon_0>0$ and $\rho_0=0$.  We put
$\dot{\rho}(x) =\delta \rho\,(x): = \frac{d}{d\epsilon}{|_{\epsilon=0}}
\rho_{\epsilon} (x)$. An isospectral deformation is said to be trivial if
$\Omega_{\epsilon} \simeq \Omega_0$ (up to isometry) for
sufficiently small $\epsilon$. A domain $\Omega_0$ is said to be
spectrally rigid if all  isospectral deformations
$\Omega_{\epsilon}$ are trivial. The domain $\Omega_0$ is called
infinitesimally spectrally rigid if $\dot{\rho} = 0$ (up to rigid motions) for all
isospectral deformations.

The variations of the eigenvalues are given by  Hadamard's
variational formulae,
\begin{equation} \dot{\lambda}_j(0) =  \left\{ \begin{array}{ll}
\int_{\partial \Omega(0)}\; \rho \;|\partial_{\nu}
\phi_j(0)|_{\partial \Omega}|^2 d A, & Dirichlet \\
\\ \int_{\partial \Omega} \rho \; |\phi_j(0)|_{\partial \Omega(0)}|^2 \;dA, & Neumann.
\end{array}\right.
\end{equation}
Hence, the infinitesimal  deformation condition is that the right
hand sides are zero for all $j$. To normalize the problem, we
assume with no loss of generality that the deformation is volume
preserving, which implies that
\begin{equation} \int_{\partial \Omega} \rho dA = 0.
\end{equation}
Any such $\rho$ defines a volume preserving deformation of
$\Omega$.

Thus, the infinitesimal deformation is orthogonal to all boundary
traces of eigenfunctions:
\begin{equation} \dot{\lambda}_j(0) = 0 \forall j, \iff
 \left\{ \begin{array}{ll}
\int_{\partial \Omega(0)}\; \rho \;|\partial_{\nu}
\phi_j(0)|_{\partial \Omega}|^2 d A = 0, & Dirichlet \\
\\ \int_{\partial \Omega} \rho \; |\phi_j(0)|_{\partial \Omega(0)}|^2 \;dA = 0, & Neumann.
\end{array}\right.
\end{equation}
We may rewrite these conditions in terms of the boundary values of
the  wave  kernel:
\begin{equation}
E^b(t, x, x) :=   \left\{ \begin{array}{ll}
\partial_{\nu_x}\partial_{\nu_y} U(t, x, x)|_{x  \in \partial \Omega},
& Dirichlet \\  &  \\
 U(t, x, x)|_{x  \in \partial \Omega},
& Neumann\end{array}\right.
\end{equation}
as saying that
\begin{equation} \label{INTEB} \int_{\partial \Omega} E^b(t, x, x) \rho(x) dA(x)
= 0, \forall t.  \end{equation}

\subsection{Spectral rigidity of the ellipse}

\begin{mainprob} Is an ellipse  $E_{a,b}$ given by $\frac{x^2}{a^2} + \frac{y^2}{b^2} = 1$  determined by its Dirichlet or Neumann spectrum? Is it spectrally rigid, i.e.
does there exist a non-trivial isospectral deformation $\Omega_t$ of $E_{a,b}$? Since the area  and perimeter are spectral invariants it is understood
that $\Omega_t$ has fixed area and perimeter. \end{mainprob}
\bigskip

In the case of circles, the answer was shown to be `yes' by M. Kac, since the disc is extremal for the isoperimetric inequality.
But the answer remains unknown for general ellipses. Ellipses are special because they have completely integrable billiard
flows, and (since G. D. Birkhoff) are conjectured to be the only smooth plane domains with integrable billiards. 

In \cite{HZ}, H. Hezari and the author proved a partial rigidity result which breaks out of the class of real analytic domains, but
does assume that the domains all possess the symmetries of the ellipse.  Namely, ellipses are spectrally rigid among
$C^{\infty}$ plane domains with the symmetries of an ellipse. Here `rigidity' means that there do not exist smooth
curves $\Omega_t$ in the family of isometry classes of smooth domains with the spectrum of the ellipse, at least if the curves
are `non-flat' in that the Taylor expansion  of the domain at the endpoins of the axes is non-zero in the 
deformation variable $t$. This non-flatness assumption is an unexpected and  annoying technicality which it would be nice to remove.  The result is
\begin{theo} \label{RIGIDCOR} Suppose that $\Omega_0$ is an ellipse, and that
$\epsilon \to \Omega_{\epsilon}$ is a $C^\infty$   Dirichlet (or
 Neumann) isospectral deformation through $\Z_2 \times \Z_2$
symmetric $C^{\infty}$ domains. Then $\rho_{\epsilon}$ must be
flat at $\epsilon = 0$. In particular,  there exist no non-trivial
real analytic curves $\epsilon \to \Omega_{\epsilon}$ of  $\Z_2 \times \Z_2$ symmetric
$C^{\infty}$ domains with the spectrum of an ellipse.
\end{theo}

We used the Hadamard variational formula in the proof.  By taking the variation of the wave trace, we proved:

\begin{theo} \label{VARWTintro} Let $\Omega_0 \subset \R^n$ be a $C^{\infty}$
convex Euclidean domain with the property that the fixed point sets of
the billiard map are clean.  Then, for any $C^1$ variation of
$\Omega_0$ through $C^{\infty}$ domains $\Omega_{\epsilon}$, the
variation of the wave traces $\delta Tr \cos\big(  t
\sqrt{-\Delta}\big)$, with  Dirichlet (or Neumann)  boundary
conditions is a classical co-normal distribution for $t \not= m
|\partial \Omega_0|$ ($m \in \Z$) with singularities contained in
$Lsp(\Omega_0)$. For each $T \in Lsp(\Omega_0)$ for which the set
$F_T$ of periodic points of the billiard map $\beta$ of
length $T$ is a $d$-dimensional clean fixed point set consisting
of transverse reflecting rays,
 there exist non-zero constants $C_{\Gamma}$
independent of $\dot{\rho}$ such that, near $T$, the leading order
singularity is
$$\delta \;Tr \; \cos\big( t \sqrt{-\Delta}\big) \sim \frac{t}{2}\; \Re \big\{ \big( \sum_{\Gamma \subset F_T} C_{\Gamma} \int_{\Gamma} \dot{\rho}\; \gamma_1 \;  d \mu_{\Gamma}
\big) \; (t - T+ i 0^+)^{- 2 - \frac{d}{2}} \big\}, $$ modulo lower
order singularities. The sum is over the connected components $\Gamma$ of $F_T$. Here
$\delta=\frac{d}{d\epsilon}|_{\epsilon=0}$ and $\gamma_1(q, \zeta)= \sqrt {1-|\zeta|^2} $.

\end{theo}

Here $\gamma_1$ is  a certain computable function  on $B^*
\partial \Omega$.

the billiard flow and
billiard map of the ellipse are completely integrable. In
particular,  except for certain exceptional trajectories,   the
periodic points of period $T$ form a Lagrangian tori in $S^*
\Omega_0$. The exceptions are the two bouncing ball
orbits through the major/minor axes and the trajectories which
intersect the foci or glide along the boundary. The fixed point
sets of $\Phi^T$ intersect the co-ball bundle $B^*
\partial \Omega_0$ of the boundary in the fixed point sets of the
billiard map $\beta: B^* \partial \Omega_0 \to B^* \partial \Omega_0$
(for background we refer to \cite{PS,GM,HZ} for instance).
Except for the exceptional orbits, the fixed point sets are real
analytic curves. For the bouncing ball rays, the associated fixed
point sets are non-degenerate fixed points of $\beta$.

The densities $d\mu_{\Gamma}$ on the fixed point
sets of $\beta$ and its powers are very similar to the canonical densities defined
in Lemma 4.2 of \cite{DG}, and further discussed in
\cite{GM,PT}. The constants $C_{\Gamma}$ are explicit and
depend on the boundary conditions. 
Also, $F_T$ is the curve of periodic points of the billiard
map $\beta$ on $B^*
\partial \Omega$ of length $T$,  where  as above the length is defined by the length of the corresponding
billiard trajectory.  The dimension $d$ is the dimension of  $F_T$.
 In the case of the ellipse, for instance, $d = 1$; the periodic
 points of a given length form invariant curves for $\beta$.

\begin{theo} \label{RIGID} Suppose that $\Omega_0$ is an ellipse, and that
$\Omega_{\epsilon}$ is a $C^1$ Dirichlet (or Neumann) isospectral
deformation of $\Omega_0$ through $C^{\infty}$ domains with  $\Z_2
\times \Z_2$ symmetry.  I.e.  is invariant under $(x, y) \to
(\pm x, \pm y)$. Then 
$\dot{\rho} = 0$.

\end{theo}

This
implies that ellipses admit no isospectral deformations for which
the Taylor expansion of $\rho_{\epsilon} $ at ${\epsilon} = 0$ is
non-trivial. A function such as $e^{- \frac{1}{{\epsilon^2}}}$ for
which the Taylor series at $\epsilon = 0$ vanishes is called `flat' at
${\epsilon}= 0$.

Since the final step of the proof uses results of \cite{GM}, we
briefly review the description of the billiard map of the ellipse
$\Omega_0:=\frac{x^2}{a} + \frac{y^2}{b} = 1$ (with $a > b > 0$)  in
that article. In the interior, there exist for each $0 < Z \leq b$
a caustic set given by a confocal ellipse
$$\frac{x^2}{E + Z} + \frac{y^2}{Z} = 1$$
where $E = a - b$,  or for $- E < Z < 0$ by a
confocal hyperbola. Let $(q, \zeta)$ be in $B^*\partial \Omega_0$ and let $(q,\xi)$ in $S^*\Omega_0$ be the unique inward unit normal to boundary that projects to $(q,\zeta)$. The line segment $(q,r\xi)$ will be tangent to a unique confocal ellipse or hyperbola (unless it intersects the foci). We then define the function $Z(q, \zeta)$ on $B^* \partial \Omega_0$ to be the corresponding $Z$.  Then $Z$ is a $\beta$-invariant function
and its level sets $\{Z = c\}$ are the invariant curves of
$\beta$. The invariant Leray form on the level set is denoted
$du_Z$ (see \cite{GM}, (2.17), i.e. the symplectic form of $B^*\partial \Omega_0$
is $dq \wedge d \zeta = dZ \wedge du_Z$. A level set has a rotation number and the
periodic points live in the level sets with rational rotation
number. As it is explained in \cite{GM} (page 143) the Leray form $du_Z$ restricted to a connected component $\Gamma$ of $F_T$ is a constant multiple of the canonical density $d\mu_\Gamma$.


 As mentioned in the introduction, the  well-known obstruction to using trace formula calculations
 such as in Proposition \ref{VARWTintro} is multiplicity in the length
 spectrum, i.e. existence of several connected components of
 $F_T$. A higher dimensional component is not itself a
 problem, but there could exist cancellations among terms coming
 from components with different Morse indices, since the
 coefficients $C_{\Gamma}$ are complex. This problem arose earlier
 in the spectral theory of the ellipse in \cite{GM}. Their key
 Proposition 4.3 shows that there is a sufficiently large set
 of lengths $T$ for which $F_T$ has one component up to $(q, \zeta) \to (q, -\zeta)$ symmetry.
 Since it is crucial here as well, we state the relevant part:

 \begin{prop}\label{GM}  (see \cite{GM}, Proposition 4.3): Let $T_0 =
 |\partial \Omega_0|$. Then for every interval $(m T_0 - \epsilon,
 m T_0)$ for $m = 1, 2, 3, \dots$ there exist infinitely many
 periods $T \in Lsp(\Omega_0)$ for which $F_T$ is the union
 of two invariant curves which are mapped to each other by $(q,\zeta)
 \to (q, -\zeta)$. \end{prop}
Since for an isospectral deformation $\delta\; Tr \cos( t\sqrt{-\Delta})=0$, we obtain from Proposition \ref{VARWTintro} the following

\begin{cor} \label{MAIN} Suppose we have an isospectral deformation of an ellipse $\Omega_0$ with velocity $\dot{\rho}$. Then for each $T$ in Proposition \ref{GM} for which $F_T$ is the union
 of two invariant curves $\Gamma_1$ and $\Gamma_2$ which are mapped to each other by $(q, \zeta) \to (q, -\zeta)$ we have
$$\int_{\Gamma_j} \dot{\rho} \; \gamma_1 \; du_Z = 0, \qquad \quad j=1,2.$$
 \end{cor}
\begin{proof}
From Proposition \ref{VARWTintro} we get
$$\Re \big\{ \big( \sum_{j=1}^2 C_{\Gamma_j} \int_{\Gamma_j} \dot{\rho}\; \gamma_1 \;  d \mu_{\Gamma_j}
\big) \; (t - T+ i 0)^{- 2 - \frac{d}{2}} \big\}=0.$$ Since $\dot{\rho}$ and $\gamma_1$ are invariant under the time reversal map $(q,\zeta) \to (q,-\zeta)$, the two integrals are identical. Also by directly looking at the stationary phase calculations it can be shown that the Maslov coefficients $C_{\Gamma_1}$ and $C_{\Gamma_2}$ are also the same. Thus the corollary follows.
\end{proof}

\subsection{Abel transform}

The remainder of the proof of Theorem \ref{RIGID} is identical to
that of Theorem 4.5 of \cite{GM} (see also \cite{PT}). For the sake of completeness, we
sketch the proof.

\begin{prop}\label{RHODOT}  The only $\Z_2 \times \Z_2$ invariant function
$\dot{\rho}$ satisfying the equations of Corollary \ref{MAIN} is
$\dot{\rho} = 0$. \end{prop}

\begin{proof} First, we may assume $\dot{\rho} = 0$ at the
endpoints of the major/minor axes, since the deformation preserves
the $\Z_2 \times \Z_2$ symmetry and we may assume that the
deformed bouncing ball orbits will not move and are aligned with the original ones.
Thus $\dot{\rho}(\pm \sqrt{a}) = \dot{\rho}(\pm \sqrt{b}) = 0$.

The Leray measure may be explicitly evaluated (see $2.18$ in \cite{GM}). By a change of
variables with Jacobian $J$, and using the symmetric properties of $\dot{\rho}$, the integrals become
\begin{equation} \label{F} A(Z) = \int_b^a \frac{\dot{\rho}(t)\; \gamma_1 \; J(t) dt}{\sqrt{t -
(b - Z)}}. \end{equation} for an infinite sequence of $Z$
accumulating at $b$.  The  function $A(Z)$ is
smooth in $Z$ for $Z$ near $b$. It vanishes infinitely often in
each interval $(b - \epsilon, b)$, hence is flat at $b$. The $k$th
Taylor coefficient at $b$ is \begin{equation} \label{FT}
A^{(k)}(b) = \int_b^a \dot{\rho}(t)\; \gamma_1 \; J(t) t^{- k -
\half} dt = 0.
\end{equation}
Since the functions $t^{-k}$ span a dense subset of $C[b, a]$, it
follows that $\dot{\rho} \equiv 0.$

\end{proof}

We refer to \S \ref{QOSCINT} for some further problems regarding the isospectral determinantion problem for ellipses. 

\subsubsection{Spectral rigidity of  convex analytic domain with the symmetries of an ellipse?}

We observe that the proof of Theorem \ref{RIGIDCOR} only  uses the hypothesis that
the domain is an ellipse in two steps. First, in Theorem \ref{VARWTintro} it uses the very
special kind of singularities of the wave trace due to the foliation of $B^* \partial E$ by invariant
curves for the billiard flow. Then in Proposition \ref{RHODOT}  it uses the explicit formula \eqref{F}
for the integrals over the invariant curves. In the case of general convex smooth domains, there exist
invariant sets for the billiard map due to periodic  creeping orbits along the boundary, but they are much more complicated
than in the completely integrable case of the ellipse.
Still, it may be possible to analyze the analogous 
integrals $A(Z)$.

\section{\label{PHASE} Inverse scattering, phase shifts, resonance poles}

In this section we consider the inverse phase shift problem for semi-classical Schr\"odinger operators \eqref{SO}
on $\R^d$ (with the Euclidean Laplacian).
The scattering matrix  $$S_{h}: L^2(\mathbb{S}^{d-1}) \to L^2(\mathbb{S}^{d-1})$$  can be defined in terms of
generalized eigenfunctions as follows.  For $\phi_{in} \in
C^{\infty}(\mathbb{S}^{d-1})$, there is a unique solution to $H
u = 0$ satisfying
\begin{equation}
  \label{eq:generalizedEfn}
  u = r^{-(d-1)/2} \left( e^{-i \sqrt{E} r / h} \phi_{in}(\omega) + e^{i
    \sqrt{E} r/ h} \phi_{out}(-\omega) \right) + O(r^{-(d + 1)/2}).
\end{equation}
 By definition
\begin{equation}\label{eq:scmatrixdef}
S_{h}(\phi_{in}) := e^{i\pi(d-1)/2}\phi_{out}.
\end{equation}
Below, we will refer to $\phi_{in}$ and $\phi_{out}$ as the
\textbf{incoming} and \textbf{outgoing data} of $u$.
It is not hard to show that
\begin{equation}
  \label{eq:Eisone}
  S_{h, V}(E) = S_{\wt{h}, \wt{V}}(1),
\end{equation}
where $\wt{h} = h / \sqrt{E}, \wt{V} = V / E$, and $S_{h', V'}(E')$
denotes the scattering matrix for $(h')^{2}\Delta + V' - E'$.  Using
\eqref{eq:Eisone}, we
may assume that $E = 1$.

Phase shifts are the eigenvalues for the  eigenvalue problem 
\begin{equation} \label{EVSh} S_h \phi_{h, n}  = e^{i \beta_{h, n}} \phi_{h, n}.  \end{equation} 

\begin{mainprob} Can one determine the potential $V$ (in some reasonable class) from the
phase shifts $\{e^{i \beta_{h, n}} \}_{n=1}^{\infty}$ for fixed energy but as $h$ varies?
\end{mainprob}

The simplest problem is to assume maximal symmetry, i.e. that the potential is radial. This effectively reduces
the problem to one dimension. The scattering matrix then commutes with the action of $SO(d)$ and $\Delta_{S^{d-1}}$
and thus has just one eigenvalue in each eigenspace of the Laplacian. When $d = 3$, the  eigenvalues of  $\Delta_{S^{2}}$ on
the sphere are known as the total angular momenta, and thus the phase shifts are traditionally denoted  $\delta_{\ell}(k)$
where the angular momentum is $\ell(\ell + 1)$ and $k = h^{-1}$. 

The question of finding a radial potential $V$ from its phase shifts has been much studied since the $1960's$. The phase
shifts are closely related to the Jost functions $f_{\ell}(k) = |f_{\ell}(k)| e^{- i \delta_{\ell}(j)}. $ The first
issue is whether \eqref{SO} has any $L^2$-eigenfunctions or `bound states'.  In classical scattering the corresponding
issue is whether there exist trapped Hamilton trajectories or whether they all escape to infinity. It one knows all of
the eigenvalues $E_n(\hbar)$ and the phase shifts $\delta_{\ell}(k)$ for all $k \geq 0$ then one can construct
$|f_{\ell}(k)|$. Equivalently one knows the ratio $f_{\ell}(-k)/ f_{\ell}(k)$. The problem of determining $f_{\ell}(k)$
from the ratios is a Riemann-Hilbert problem. Using the Gelfand-Levitan solution of the inverse spectral problem in
dimension one, it has been shown that if there are $N$ eigenvalues then there exists an N-parameter family of
potentials with the same phase shifts as $k$ varies and $\ell$ is fixed  and same eigenvalues \cite{N,N2,N3}. In 
particular, the potential is uniquely determined if \eqref{SO} has no $L^2$ eigenfunctions. 
If instead one fixes $k$ and lets $\ell $ vary, one can also solve the problem but not uniquely.  
 
To our knowledge, the problem of determining $V$ from the $\beta_{h, n}$ in non-radial cases is open and the
analogues of the methods described above for Laplacians on compact manifolds have not been discussed.

A closely related problem is to study the Laplacian $\Delta$ of the Euclidean metric in the exterior of a convex obstacle
$\ocal$. Thus, we let $\Omega = \R^d \backslash \ocal$ and consider the scattering matrix $S_h$ for $h^2 \Delta$
on $L^2(\Omega)$ with Dirichlet or Neumann boundary conditions. 

\bigskip

\begin{mainprob}\;\;Can one determine $\Omega$ from the fixed energy phase shifts $e^{i\beta_{h,n}} $?. \end{mainprob}
\bigskip

The inverse problem for obstacles is trivial in the radial case but very non-trivial in general. There exist results
on determining the potential from the scattering kernel (Theorem 2 of \cite{Maj1976, G3}), but to our knowledge the phase shift problem has not
been studied. The method of \cite{G3} is to determine the so-called sojourn times from the full scattering kernel
$S_h(\omega, \theta)$ as $h \to 0$ and to show that it dtermines the Gaussian curvature of the surface. In the
case of a convex obstacle, the curvature function determines the obstacle, i.e. the shape of the scatterer. 
The question of how much of the shape of the scatterer one might recover from the phase shifts alone does
not seem to have been studied in the non-radial case.

As J. Gell-Redman emphasized to the author, the  so-called interior-exterior duality establishes a relation between the interior Dirichlet eigenvalues and the
exterior Dirichlet phase shifts \cite{EP}. From the phase shifts one can determine the interior Dirichlet
eigenvalues. Hence positive results on the inverse problem for domains imply positive results for phase shifts. 

We refer to \cite{Maj1976,ChS,CT,KKS,MT,N,N2,N3,Sab} for background on inverse results in dimensions $> 1$
and to   \cite{GHZ,DGHH2013} for some recent results on the distribution of phase shifts. Some further references in dimension one
are given in the introduction.

\section{\label{PROBS}Problems and partial solutions}

In this section we state some problems which are at least partly open. They are intended both to give
a sense of the current state of the art and also to suggest problems on which it seems likely one can make
progress.

The following is probably the most natural:

\bigskip

\begin{mainprob}  Is a convex analytic domain spectrally rigid in the class $\ccal 
\acal$ of convex smooth  domains?  More difficult:  Is a  $C^{\infty}$ convex  domain spectrally rigid? \end{mainprob}\bigskip

\bigskip
There are several possible approaches. One is to develop the proof of Theorem \ref{RIGIDCOR}. Only in
the last step is the geometry of the ellipse used. It will change radically for a general convex domain. 
In particular one must use the invariant sets which arise in general. 

In view of the number of inverse spectral results which use the assumption that the domains are real
analytic, we pose:

\begin{mainprob}  Suppose $\Omega_1$ and $\Omega_2$ are isospectral plane domains. If $\Omega_1$ is real analytic,
is $\Omega_2$ real analytic?  Is this true if we assume both domains have the symmetry of an ellipse?\end{mainprob} \bigskip

The reason for posing the problem is that, when a domain has the symmetry of an ellipse, one can determine
the Taylor expansion of the domain at the endpoints of the two axes from the spectrum. Hence if the expansion
converges for the first domain, the Taylor series converges for the second domain. But it might not actually equal
the second domain.   More precisely, we represent the top and bottom of the  domains as graphs
$y = f_{\pm}(x)$ over one of the axes and then the Taylor expansion of $f^j_{\pm}(x)$ $j = 1,2$  is determined by the spectrum.
Under the symmetry assumption $f^j_- = - f^j_+$. For instance, in the case of an ellipse $E_{a,b}$, $f_+(x) =
b \sqrt{1 - \frac{x^2}{a^2}}. $ We note that the radius of convergence of $f_{\pm}$ at $x = 0$ is $|x| < a$, i.e. until
the graph ceases to be a graph by turning vertical.  The length of the axes is a spectral invariant
and therefore it must coincide with the first domain.

The question arises whether $\Omega_2$ is locally defined by $f^1_{\pm}$. It is possible that $f^2_{\pm} = 
f^1_{\pm} + \psi_{\pm}$ where $\psi$ vanishes to infinite order at the endpoints of the bouncing ball orbit. 
The wave invariants still coincide but it is not clear that the wave trace itself does. One might need ``exponentially
small corrections" to the wave trace to determine that, or else a proof that the wave trace expansion converges. 
It is not clear that this is possible for isospectral surfaces. A simpler problem is

\begin{mainprob}  Suppose $\Omega_1$ and $\Omega_2$ are isospectral plane domains. If $\Omega_1$ is real analytic,
is $\Omega_2$ locally real analytic at the endpoints of a bouncing ball orbit?  Is this true if we assume both domains have the symmetry of an ellipse?\end{mainprob} \bigskip

If so,  we then ask:  Suppose that $\Omega_2$ is isospectral to a real analytic $\Omega_1$ and that 
$\partial \Omega_2 = \partial \Omega_1$ on some interval $x \in (-\epsilon, \epsilon)$ around the endpoints
of a bouncing ball orbit. Must then $\Omega_2 = \Omega_1?$

Another question concerns the normal form in the analytic case. 

\begin{mainprob} Does the quantum Birkhoff normal form at a hyperbolic geodesic $\gamma$ converge if the metric is
real analytic? Does the conjugating map (unitary operator) converge?\end{mainprob}

It is proved in \cite{Mo} that this is true for the classical Birkhoff normal forms in dimension 2. Some hints that the answer is
positive for quantum normal forms  in dimension 2 are given in \cite{Sj,A,Sa,PM,I}. If the answer is `yes' then the 
manifolds are locally Fourier-isospectral near $\gamma$. In particular the geodesic flows are locally symplectically
equivalent.  

\begin{mainprob} Do isospectral plane domains have symplectically equivalent billiard maps? If a domain
$\Omega$ has the same eigenvalues as the ellipse $E_{a,b}$, does it have a completely integrable billiard map? \end{mainprob}
\bigskip

Results of Siburg \cite{S} and of the author and G. Forni show that the billiard flow would be $C^0$ conjugate to that
of the ellipse, at least if one has an isospectral deformation. But there is a long distance between $C^0$ and smooth
conjugacy. 

\subsection{Problems on surfaces}

\begin{mainprob} If $g$ is a metric on $S^2$ such that the distinct eigenavalues $\lambda_k^*$  of $\Delta_g$ 
have the same multiplicities $m(\lambda_k^*) = 2k + 1$ as the standard metric $g_0$, is $g = g_0$?\end{mainprob}

\subsection{\label{QOSCINT} Questions about oscillatory integrals}

\begin{mainprob} If an oscillatory integral $\int_{\R^n} a e^{i k \phi} dx$ whose phase has a single critical
point at $x = 0$ has the stationary phase asymptotics of a Morse function with a single critical point, is it a Morse function?
Similarly, if it has the stationary phase asymptotics of a Bott-Morse phase function with a non-degenerate critical
manifold, is it Bott-Morse? \end{mainprob}

\bigskip

The reason for posing  this problem is that wave trace invariants are obtained by applying stationary phase
to a parametrix for the wave group, and one only knows the wave trace invariants explicitly if the closed
geodesics are non-degenerate (i.e. if the phase of the integral is Morse) or `clean' (if the phase is Bott-Morse).
If a second domain or metric has the same wave trace invariants as the first, then its wave trace expansion
at a closed geodesic is the same as one for first, given by stationary phase for a Morse function. It would simplify
the inverse problem considerably if we could know that that the closed geodesic of the unknown metric or domain
was also Morse (or Bott-Morse). For instance in the case of the ellipse, the phase functions are Bott-Morse.

As a potential application, we pose:

\begin{mainprob} Suppose that $\Omega$ is isospectral to an ellipse. Do the periodic orbits of  the billiard map $\beta$ on   $B^* \partial \Omega$
form  smooth invariant curves?  Similarly, if $(S^2, g)$ is a metric on $S^2$ which is isospectral to a convex surface of revolution, 
do the periodic geodesics form smooth invariant Lagrangian tori? \end{mainprob}

The point is that the wave trace invariants are the same as for the ellipse. resp. $(S^2, g)$ (here we ignore
issues of length spectral multiplicity). As we see from Theorem \ref{VARWTintro}, the wave trace has the
singularities of a Bott-Morse oscillatory integral with clean fixed point sets of dimension 2 (in $ S^*_g M$) resp. 1
(in $B^*\partial \Omega$).  A positive answer to the Problem about oscillatory integrals with the stationary
phase asymptotics of a Bott-Morse function would imply that the phase has to be Bott-Morse for the unknown
metric or boundary and therefore would be smooth invariant submanifolds. 

A related more ambitious question is:

\begin{mainprob} Suppose that $\Omega$ is isospectral to an ellipse. Is $B^* \partial \Omega$ foliated by smooth
invariant curves for $\beta$? Similarly, if $(S^2, g)$ is a metric on $S^2$ which is isospectral to a surface of revolution, is
$S^*_g S^2$ foliated by smooth Lagrangian tori invariant under the geodesic flow? \end{mainprob}

\end{document}